\documentclass[preprint,3p]{elsarticle}

\usepackage{hyperref}
\usepackage{amsmath}
\usepackage{amsfonts,amssymb}
\usepackage{graphicx}
\usepackage{bm}
\graphicspath{{./Figures/}}
\usepackage{color}
\usepackage{ulem}
\usepackage{multirow}
\usepackage{multicol}
\usepackage{algorithm}
\usepackage{algpseudocode}
\usepackage{algorithmicx}
\usepackage{hyperref}
\usepackage{upquote}
\usepackage{cprotect}

\usepackage[showonlyrefs]{mathtools}
\usepackage[dvipsnames]{xcolor}

\usepackage[final,defaultcolor=red]{changes}

\DeclareMathAlphabet{\mathbf}{OT1}{cmr}{bx}{it}
\newcommand{\va}{{\mathbf a}}
\newcommand{\vb}{{\mathbf b}}

\newcommand{\ve}{{\mathbf e}}

\renewcommand{\d}{\,\mathrm{d}}
\newcommand{\dmu}{\d\mu(t)}

\newcommand{\R}{\mathbb{R}}

\newcommand{\Rnn}{\mathbb{R}^{n \times n}}
\newcommand{\C}{\mathbb{C}}
\newcommand{\calD}{\mathcal{D}}
\newcommand{\calO}{\mathcal{O}}

\newcommand{\Cnn}{\mathbb{C}^{n \times n}}

\newcommand{\lmax}{\lambda_{\textnormal{max}}}
\newcommand{\lmin}{\lambda_{\textnormal{min}}}

\newcommand{\cond}{\textnormal{cond}}

\newcommand{\condabs}{\cond_{\textnormal{abs}}}
\newcommand{\condtwoabs}{\cond_{\textnormal{abs}}^{[2]}}

\DeclareMathOperator{\spec}{spec}

\newtheorem{theorem}{Theorem}
\newtheorem{lemma}[theorem]{Lemma}
\newtheorem{proposition}[theorem]{Proposition}
\newtheorem{corollary}[theorem]{Corollary}
\newproof{proof}{Proof}
\newtheorem{remarksimple}[theorem]{Remark}
\let\oldremarksimple\remarksimple
\renewcommand{\remarksimple}{\oldremarksimple\normalfont}
\newenvironment{remark}{\begin{remarksimple}}{\hfill$\diamond$\end{remarksimple}}
\newtheorem{experimentsimple}[theorem]{Experiment}
\let\oldexperimentsimple\experimentsimple
\renewcommand{\experimentsimple}{\oldexperimentsimple\normalfont}
\newenvironment{experiment}{\begin{experimentsimple}}{\hfill$\diamond$\end{experimentsimple}}

\newtheorem{examplesimple}[theorem]{Example}
\let\oldexamplesimple\examplesimple
\renewcommand{\examplesimple}{\oldexamplesimple\normalfont}

\def\Fd{Fr\'{e}chet derivative}

\usepackage{pgfplots}
\usepgfplotslibrary{groupplots}
\usetikzlibrary{external}

\pgfkeys{/pgf/images/include external/.code={\includegraphics{#1}}}

\usetikzlibrary{decorations.markings}
\makeatletter
\tikzset{
  nomorepostactions/.code={\let\tikz@postactions=\pgfutil@empty},
  mymark/.style 2 args={decoration={markings,
    mark= between positions 0 and 1 step (1/9)*\pgfdecoratedpathlength with{%
        \tikzset{#2,every mark}\tikz@options
        \pgfuseplotmark{#1}%
      },  
    },
    postaction={decorate},
    /pgfplots/legend image post style={
        mark=#1,mark options={#2},every path/.append style={nomorepostactions}
    },
  },
}
\makeatother

\def\addlegendimage{\csname pgfplots@addlegendimage\endcsname}

\definecolor{color_marcel1}{named}{Cerulean}
\definecolor{color_marcel2}{named}{BurntOrange}
\definecolor{color_marcel3}{named}{OliveGreen}
\definecolor{color_marcel4}{named}{Orchid}
\pgfplotscreateplotcyclelist{list_std}{%
color_marcel1,line width=1.25pt, mark=x, solid\\
color_marcel2,line width=1.25pt, mark=o, solid\\
color_marcel3,line width=1.25pt, mark=diamond, solid\\
color_marcel4,line width=1.25pt, mark=square, solid\\
}

\pgfplotscreateplotcyclelist{list_mod3}{%
color_marcel2,line width=1.25pt, mark=o, solid\\
color_marcel3,line width=1.25pt, mark=diamond, solid\\
color_marcel4,line width=1.25pt, mark=square, solid\\
}

\pgfplotscreateplotcyclelist{list_std3}{%
color_marcel1,line width=1.25pt, mark=x, solid\\
color_marcel2,line width=1.25pt, mark=o, solid\\
color_marcel3,line width=1.25pt, mark=diamond, solid\\
}

\pgfplotscreateplotcyclelist{list_std1}{%
color_marcel1,line width=1.25pt, mark=x, solid\\
}

\pgfplotsset{compat=1.17}

\journal{}

\begin{document}

\begin{frontmatter}

\title{Integral representations for higher-order Fr\'echet derivatives of matrix functions: Quadrature algorithms and new results on the level-2 condition number}

\author{Marcel Schweitzer}
\address{School of Mathematics and Natural Sciences, Bergische Universit\"at Wuppertal, 42097 Wuppertal, Germany}

\ead{marcel@uni-wuppertal.de}

\begin{abstract}
We propose an integral representation for the higher-order \Fd\ of analytic matrix functions $f(A)$ which unifies known results for the first-order \Fd\ of general analytic matrix functions and for higher-order \Fd{}s of $A^{-1}$. We highlight two applications of this integral representation: On the one hand, it allows to find the exact value of the level-2 condition number (i.e., the condition number of the condition number) of $f(A)$ for a large class of functions $f$ when $A$ is Hermitian. On the other hand, it also allows to use numerical quadrature methods to approximate higher-order \Fd{}s. We demonstrate that in certain situations---in particular when the derivative order $k$ is moderate and the direction terms in the derivative have low-rank structure---the resulting algorithm can outperform established methods from the literature by a large margin.
\end{abstract}


\end{frontmatter}

%\linenumbers

\section{Introduction}

Matrix functions
$f \colon \mathbb{C}^{n \times n} \rightarrow \mathbb{C}^{n \times n}$
play an important role in many areas of applied mathematics. Among many other applications, the matrix exponential, $f(A) = e^A$, arises in network analysis~\cite{EstradaHigham2010} and exponential integrators~\cite{HochbruckLubich1997,HochbruckOstermann2010, HochbruckLubichSelhofer1998}, while inverse fractional powers occur in lattice quantum chromodynamics~\cite{VanDenEshofFrommerLippertSchillingVanDerVorst2002,Neuberger1998} and statistical sampling~\cite{IlicTurnerSimpson2010}.

The \Fd\ of a matrix function is defined as the unique operator
$L_f(A, \cdot)\colon \mathbb{C}^{n\times n} \rightarrow \mathbb{C}^{n\times n}$
that is linear in its second argument and,
for any matrix $E \in \mathbb{C}^{n \times n}$,
satisfies
\begin{equation}
  \label{eq:fd}
  f(A + E) - f(A) - L_f(A,E) = o(\|{E}\|),
\end{equation}
where $\|\cdot\|$ denotes any matrix norm and the remainder term $o(\|E\|)$, when divided by $\|E\|$, tends to zero as $\|E\| \rightarrow 0$. Based on~\eqref{eq:fd}, the $k$th \Fd\ of $f$ at $A$ can be recursively defined as the unique multilinear function $L_f^{(k)}(A, \cdot, \dots, \cdot)$ of the matrices $E_i, i = 1,\dots,k$ that satisfies
% \begin{eqnarray}
% \|L_f^{(k-1)}(A+E_k, E_1,\dots,E_{k-1}) &-& L_f^{(k-1)}(A, E_1, \dots, E_{k-1}) \nonumber\\
%                                         &-& L_f^{(k)}(A, E_1, \dots, E_k)\| = o(\|E_k\|), \label{eq:kth_order_frechet}
% \end{eqnarray}
\begin{equation*}
L_f^{(k-1)}(A+E_k, E_1,\dots,E_{k-1}) - L_f^{(k-1)}(A, E_1, \dots, E_{k-1}) - L_f^{(k)}(A, E_1, \dots, E_k) = o(\|E_k\|), \label{eq:kth_order_frechet}
\end{equation*}
where $L_f^{(1)}(A,E) := L_f(A, E)$ is the usual (first-order) \Fd\ defined in~\eqref{eq:fd}; see~\cite{HighamRelton2014}.

An important application of the \Fd\ is calculating the condition number of computing $f(A)$. Precisely, denoting by $\cond_{\text{abs}}(f, A)$ the absolute condition number of computing $f(A)$, we have
\begin{equation}\nonumber
  \label{eq:fd_condition}
  \mathrm{cond}_{\text{abs}}(f, A) =
  \lim_{\varepsilon \rightarrow 0} \sup_{\|E\| \leq \varepsilon}
  \frac{\|f(A+E)-f(A)\|}{\varepsilon} = \max_{\|E\|=1}\|L_f(A,E)\|;
\end{equation}
see~\cite[Chapter~3]{Higham2008}. In turn, the second-order \Fd\ can be used to compute the \emph{level-two condition number} of a matrix function, i.e., the condition number of the condition number. In particular, the absolute level-two condition number is given by
$$\cond_{\text{abs}}^{(2)}(f, A) := \lim\limits_{\varepsilon \rightarrow 0}\sup\limits_{\|Z\|\leq \varepsilon} \frac{|\cond_{\text{abs}}(f, A+Z) - \cond_{\text{abs}}(f, A)|}{\varepsilon}.$$
In~\cite{HighamRelton2014}, Higham and Relton relate the level-two condition number (with respect to the Frobenius matrix norm) to the second order Fr\'echet derivative, using the relation
\begin{equation}\label{eq:lvl2_kronecker}
\cond_{\text{abs}}^{(2)}(f, A) \leq \max\limits_{\|E_2\|_F=1}\max\limits_{\|E_1\|_F=1} \|L_f^{(2)}(A, E_1, E_2)\|_F = \|K_f^{(2)}(A)\|_2,
\end{equation}
where $\|\cdot\|_F$ and $\|\cdot\|_2$ denote the Frobenius and spectral matrix norm, respectively, and $K_f^{(2)}(A)$ is the \emph{Kronecker matrix}, a representation of the second-order \Fd\ as an $n^4 \times n^2$ matrix; see~\cite[Section~4 \& 5]{HighamRelton2014}. In addition to this application, higher-order \Fd{}s also have uses in the solution of nonlinear equations on Banach spaces; see~\cite{AmatBusquierGutierrez2003}.

The remainder of this paper is organized as follows. In Section~\ref{sec:integral_representation}, we derive an integral representation for the higher-order \Fd\ of analytic functions. Section~\ref{sec:lvl2_condition_number} then deals with finding the level-2 condition number of $f(A)$: We first recap some known formulas and bounds from the literature and then discuss how the integral representation derived in Section~\ref{sec:integral_representation} can be used to extend and improve these results. Section~\ref{sec:algorithms} is devoted to computational methods for the higher-order \Fd: We review some established algorithms for this task and then propose to use our new integral representation as the basis of a quadrature method and assess the computational complexity of the resulting algorithm. Numerical experiments are reported in Section~\ref{sec:numerical_experiments}. Some concluding remarks and directions for future research are given in Section~\ref{sec:concluding_remarks}.

\section{An integral representation for the higher-order \Fd}\label{sec:integral_representation}
In this section, we derive new integral representations for the higher-order \Fd\ which form the basis of the new results and algorithms that are covered in Section~\ref{sec:lvl2_condition_number} and~\ref{sec:algorithms}. We begin by considering analytic functions represented by the Cauchy integral formula and then comment on several related integral representations.

\subsection{Analytic functions represented by the Cauchy integral formula}\label{subsec:cauchy_integral_formula}
Assume that $f$ is analytic on and inside a contour $\Gamma$ that winds around $\spec(A)$, the spectrum of $A$, exactly once. Then, using the Cauchy integral formula, we have
\begin{equation}\label{eq:cauchy_integral_fA}
f(A) = \frac{1}{2\pi i} \int_\Gamma f(\zeta) (\zeta I - A)^{-1} \d \zeta,
\end{equation}
where $I$ denotes the $n\times n$ identity matrix. It is known~\cite{Higham2008,KandolfRelton2017,KandolfKoskelaReltonSchweitzer2021} that in this case, the first-order \Fd\ is given by
\begin{equation}\label{eq:frechet_derivative_cauchy_integral}
L_f(A,E) = \frac{1}{2\pi i} \int_\Gamma f(\zeta)(\zeta I - A)^{-1}E(\zeta I - A)^{-1}\d \zeta.
\end{equation}

We now derive a generalization of formula~\eqref{eq:frechet_derivative_cauchy_integral} for arbitrary derivative order $k$. This derivation is based on the following basic result on the $k$th order \Fd\ of the resolvent.

\begin{lemma}\label{lem:kth_fd_resolvent}
Let $\zeta \notin \spec(A)$. Then
\begin{equation*}\label{eq:kth_fd_resolvent}
L_{(\zeta-z)^{-1}}^{(k)}(A, E_1,\dots,E_k) = \sum\limits_{\pi \in S_k} (\zeta I-A)^{-1}E_{\pi(1)}(\zeta I-A)^{-1}E_{\pi(2)}(\zeta I-A)^{-1} \cdot{}\dots{}\cdot E_{\pi(k)}(\zeta I-A)^{-1},
\end{equation*}
where $S_k$ denotes the \emph{symmetric group of degree $k$}, i.e., the set of all permutations of $\{1,\dots,k\}$.
\end{lemma}
\begin{proof}
We prove the statement by induction. Using the Neumann series, for $\|E_1\|$ small enough, we have
\begin{eqnarray}
(\zeta I - (A+E_1))^{-1} &=& (\zeta I - A)^{-1}(I - E_1(\zeta I-A)^{-1})^{-1} \nonumber\\
 &=& (\zeta I - A)^{-1}\sum\limits_{\ell=0}^\infty (E_1(\zeta I-A)^{-1})^\ell \nonumber\\
&=& (\zeta I-A)^{-1} + \sum\limits_{\ell=1}^\infty (\zeta I-A)^{-1} (E_1(\zeta I-A)^{-1})^\ell \nonumber\\
&=& (\zeta I-A)^{-1} - (\zeta I-A)^{-1}E_1(\zeta I-A)^{-1} + o(\|E_1\|), \label{eq:neumann_series}
\end{eqnarray}
which implies
\begin{equation*}\label{eq:first_order_frechet_inverse}
    L^{(1)}_{(\zeta-z)^{-1}}(A, E_1) = (\zeta I-A)^{-1}E_1(\zeta I-A)^{-1},
\end{equation*}
so that the assertion of the lemma is valid for $k = 1$. Now, for $k > 1$, we inductively have
\begin{eqnarray}
& & L_{(\zeta-z)^{-1}}^{(k-1)}(A + E_k, E_1, \dots, E_{k-1}) \nonumber\\ 
&=& \sum\limits_{\pi \in S_{k-1}} (\zeta I-(A+E_k))^{-1} E_{\pi(1)} (\zeta I-(A+E_k))^{-1} E_{\pi(2)}(\zeta I-(A+E_k))^{-1} \cdot{}\dots{}\nonumber\\
& &\quad\quad\quad\quad\dots{}\cdot E_{\pi(k-1)}(\zeta I-(A+E_k))^{-1}\label{eq:proof_resolvent1}
\end{eqnarray}
Inserting~\eqref{eq:neumann_series} for $(\zeta I-(A+E_k))^{-1}$ into~\eqref{eq:proof_resolvent1} and collecting terms up to first order in $E_k$ gives
\begin{eqnarray}
& & L_{(\zeta-z)^{-1}}^{(k-1)}(A + E_k, E_1, \dots, E_{k-1}) \nonumber\\ 
&=& \sum\limits_{\pi \in S_{k-1}} (\zeta I-A))^{-1} E_{\pi(1)} (\zeta I-A)^{-1} E_{\pi(2)}(\zeta I-A)^{-1} \cdot{}\dots{}\cdot E_{\pi(k-1)}(\zeta I-A)^{-1}\nonumber\\
& & + \sum\limits_{\pi \in S_{k}} (\zeta I-A))^{-1} E_{\pi(1)} (\zeta I-A)^{-1} E_{\pi(2)}(\zeta I-A)^{-1} \cdot{}\dots{}\cdot E_{\pi(k)}(\zeta I-A)^{-1} + o(\|E_k\|)\nonumber\\
&=&  L_{(\zeta-z)^{-1}}^{(k-1)}(A, E_1, \dots, E_{k-1}) + \sum\limits_{\pi \in S_{k}} (\zeta I-A))^{-1} E_{\pi(1)} (\zeta I-A)^{-1} E_{\pi(2)}(\zeta I-A)^{-1} \cdot{}\dots{}\nonumber\\
& & \dots{}\cdot E_{\pi(k)}(\zeta I-A)^{-1} + o(\|E_k\|).\label{eq:proof_resolvent2}
\end{eqnarray}
Inserting relation~\eqref{eq:proof_resolvent2} into the recursive definition~\eqref{eq:kth_order_frechet} completes the proof of the lemma.\hfill\qed
\end{proof}

From Lemma~\ref{lem:kth_fd_resolvent}, we directly obtain an integral representation for the \Fd\ of an analytic function.

\begin{theorem}\label{the:integral_representation_Lk}
Let $f$ be analytic on and inside a contour $\Gamma$ that winds around $\spec(A)$ exactly once. Then
\begin{equation}\label{eq:integral_representation}
L_f^{(k)}(A, E_1, \dots E_k) = \frac{1}{2\pi i} \int_\Gamma \sum\limits_{\pi \in S_k} f(\zeta) M_\pi(\zeta; A, E_1,\dots, E_k) \d \zeta,
\end{equation}
where 
\begin{equation}\label{eq:integrand_M}
M_\pi(\zeta; A, E_1,\dots, E_k) = (\zeta I - A)^{-1}E_{\pi(1)}(\zeta I - A)^{-1}E_{\pi(2)}(\zeta I - A)^{-1}\cdots E_{\pi(k)}(\zeta I - A)^{-1}.
\end{equation}
\end{theorem}
\begin{proof}
The assertion of the theorem follows directly from Lemma~\ref{lem:kth_fd_resolvent} upon noting that the relation
\begin{equation*}
    L_f^{(k)}(A,E_1,\dots,E_k) = \frac{1}{2\pi i} \int_\Gamma f(\zeta) L^{(k)}_{(\zeta-z)^{-1}}(A,E_1,\dots,E_k)\d \zeta.
\end{equation*}
holds.\hfill\qed
\end{proof}

\subsection{Related integral representations}\label{subsec:other_integral_formulas}
There are several important classes of functions which have similar integral representations involving resolvents. Of particular importance is the class of Stieltjes functions. A Stieltjes function is a function that can be represented as 
\begin{equation}\label{eq:stieltjes_function}
f(z) = \int_0^\infty \frac{\dmu}{t+z},
\end{equation}
where $\dmu$ is a nonnegative measure on $[0, \infty)$. Important examples of Stieltjes functions appearing in applications are inverse fractional powers $f(z) = z^{-\alpha}, \alpha \in (0,1)$ or $f(z) = \log(1+z)/z$. For further examples and an overview of theoretical properties of Stieltjes functions, we refer the reader to~\cite{AlzerBerg2002,Berg2007,Henrici1977} and the references therein.

When $f$ is a Stieltjes function, its (higher-order) \Fd{}s also admit integral representations.
\begin{corollary}\label{cor:stieltjes}
Let $f$ be a Stieltjes function and let $\spec(A) \cap (-\infty, 0] = \emptyset$. Then
\begin{equation}\label{eq:integral_representation_stieltjes}
L_f^{(k)}(A, E_1, \dots E_k) = -\int_0^\infty \sum\limits_{\pi \in S_k} M_\pi(-t; A, E_1,\dots, E_k) \dmu
\end{equation}
where $M_\pi(\ \cdot\ ; A, E_1,\dots, E_k)$ is defined in~\eqref{eq:integrand_M}.
\end{corollary}
\begin{proof}
The formula~\eqref{eq:integral_representation_stieltjes} can be found by essentially repeating all the steps in the proofs of Lemma~\ref{lem:kth_fd_resolvent} and Theorem~\ref{the:integral_representation_Lk} adapted to the integral representation~\eqref{eq:stieltjes_function}. To avoid unnecessary repetition, we instead sketch how one can find~\eqref{eq:integral_representation_stieltjes} starting from~\eqref{eq:integral_representation}. Any Stieltjes function is analytic outside $(-\infty,0]$. Thus, when $\spec(A) \cap (-\infty,0]\added{ = \emptyset}$, we can find a contour $\Gamma$ that allows to represent $f(A)$ and thus $L_f^{(k)}(A,E_1,\dots,E_k)$ by the Cauchy integral formula. Then, inserting the Stieltjes representation~\eqref{eq:stieltjes_function} of $f$ yields
\begin{eqnarray*}
L_f^{(k)}(A, E_1, \dots E_k) &=& \frac{1}{2\pi i} \int_\Gamma \sum\limits_{\pi \in S_k} f(\zeta) M_\pi(\zeta; A, E_1,\dots, E_k) \d \zeta \nonumber\\
&=& \frac{1}{2\pi i} \int_\Gamma \sum\limits_{\pi \in S_k} \int_0^\infty\frac{\dmu}{t+\zeta} M_\pi(\zeta; A, E_1,\dots, E_k) \d \zeta \nonumber\\
&=& \int_0^\infty \sum\limits_{\pi \in S_k} \frac{1}{2\pi i} \int_\Gamma \frac{1}{t+\zeta} M_\pi(\zeta; A, E_1,\dots, E_k) \d \zeta \dmu \nonumber\\
&=& -\int_0^\infty \sum\limits_{\pi \in S_k} M_\pi(-t; A, E_1,\dots, E_k) \dmu.
\end{eqnarray*}
The last equality follows from the residue theorem upon noting that the integrand has a single pole at $\zeta = -t$.\hfill\qed
\end{proof}

Another highly related class of functions consists of those functions which can be written as $f(z) = z g(z)$ with $g$ a Stieltjes function. According to the examples for Stieltjes functions we have given above, we have that $f(z) = z^{\sigma} = zz^{1-\sigma}, \sigma \in (0,1)$ or $f(z) = \log(1+z) = z\log(1+z)/z$ belong to this class. Another example is $W_0(z)$, the primary branch of the Lambert W function; see~\cite{KaluginJeffreyCorlessBorwein2012}. As matrix square roots and logarithms as well as the Lambert W function play an important role in many applications, this class of functions warrants a detailed study.

\begin{theorem}\label{the:stieltjes_times_z}
Let $f(z) = z g(z)$, where $g$ is a Stieltjes function~\eqref{eq:stieltjes_function} and let $\spec(A) \cap (-\infty, 0] = \emptyset$. Then
\begin{equation*}\label{eq:integral_representation_stieltjes_times_z}
L_f^{(k)}(A, E_1, \dots E_k) = \int_0^\infty t  \cdot \sum\limits_{\pi \in S_k} M_\pi(-t; A, E_1,\dots, E_k) \dmu
\end{equation*}
where $M_\pi(\ \cdot\ ; A, E_1,\dots, E_k)$ is defined in~\eqref{eq:integrand_M}.
\end{theorem}
\begin{proof}
According to the definition of $f$, we have $f(A + E_1) = Ag(A+E_1) + E_1g(A+E_1)$. Now, inserting the formula~\eqref{eq:integral_representation_stieltjes} for the first-order \Fd\ of $g$, we have
\begin{equation*}
    g(A+E_1) = g(A) - \int_0^\infty (tI+A)^{-1}E_1(tI+A)^{-1}\dmu + o(\|E_1\|)
\end{equation*}
which yields
\begin{eqnarray*}
    & & f(A+E_1)  \\
    &=& Ag(A) - A\int_0^\infty (tI+A)^{-1}E_1(tI+A)^{-1}\dmu + E_1g(A) + o(\|E_1\|)\\
    &=& f(A) - A\int_0^\infty (tI+A)^{-1}E_1(tI+A)^{-1}\dmu + \int_0^\infty E_1(tI+A)^{-1}\dmu + o(\|E_1\|)\\
    &=& f(A) - A\int_0^\infty (tI+A)^{-1}E_1(tI+A)^{-1}\dmu + \int_0^\infty (tI+A)(tI+A)^{-1}E_1(tI+A)^{-1}\dmu + o(\|E_1\|)\\
    &=& f(A) +\int_0^\infty (tI+A-A)(tI+A)^{-1}E_1(tI+A)^{-1}\dmu + o(\|E_1\|)\\
    &=& f(A) + \int_0^\infty t(A+tI)^{-1}E_1(tI+A)^{-1}\dmu + o(\|E_1\|).
\end{eqnarray*}
This proves the statement for $k = 1$. For $k > 1$, the statement can be obtained by following exactly the same steps as in previous proofs, as the premultiplication by $A$, which complicated the situation for $k=1$ is not present any longer in the integral representation of $L_f^{(1)}(A,E_1)$. To avoid repetition, we thus omit these obvious steps here.\hfill\qed
\end{proof}

\begin{remark}\label{rem:cardoso}
We want to remark that, for the special case $f(z) = z^{1/p}, p \in \mathbb{N}$ and $k=1$, a related (but not identical) integral representation for $L_{z^{1/p}}(A,E)$ has been discussed in~\cite{Cardoso2012}.
\end{remark}

\section{The level-2 condition number of some classes of functions of Hermitian matrices}\label{sec:lvl2_condition_number}
As explained in the introduction of the paper, an important application of the second-order Fr\'echet derivative is estimating the level-2 condition number, which can be bounded as
\begin{equation}\label{eq:upper_bound_level2_condition_number}
    \condtwoabs(f, A) \leq \max\limits_{\|Z\|=1}\ \max\limits_{\|E\|=1}\ \|L_f^{(2)}(A, E, Z)\|
\end{equation}
for any matrix norm. Typically, the Frobenius norm is used for analyzing level-2 condition numbers, so this is the case that we specifically focus on in the following. It is quite difficult to obtain any more specific expressions or bounds for $\condtwoabs(f, A)$ than~\eqref{eq:upper_bound_level2_condition_number} without additional assumptions on $f$ or $A$. In~\cite{HighamRelton2014}, the authors derive explicit expressions for the level-2 condition number of the exponential of a normal matrix and inverses of general nonsingular matrices. Additionally, lower bounds for the level-2 condition number of functions of Hermitian matrices with monotonic derivative are obtained; see~Theorem~\ref{the:bound_hr} below. Using the integral representations derived in Section~\ref{sec:integral_representation}, we can show that these lower bounds are in many cases actually exact representations for the level-2 condition number. We begin by recalling the precise result from~\cite{HighamRelton2014}.

\begin{theorem}[Theorem~5.5 in~\cite{HighamRelton2014}]\label{the:bound_hr}
Let $A \in \Cnn$ be Hermitian with eigenvalues $\lambda_i$ and let $f: \R \rightarrow \R$ be such that $f(A)$ is defined and $f$ has a strictly monotonic derivative. Then in the Frobenius norm,
\begin{equation}\label{eq:condabs1}
\condabs(f,A) = \max\limits_{i} |f^\prime(\lambda_i)|
\end{equation}
Moreover, if the maximum in~\eqref{eq:condabs1} is attained for a unique $i$, say $i = k$, then
\begin{equation}\label{eq:condabs2}
\condtwoabs(f, A) \geq |f^{\prime\prime}(\lambda_k)|.
\end{equation}
\end{theorem}

Note that when $f$ is a Stieltjes function~\eqref{eq:stieltjes_function}, its derivatives are given by
\begin{equation}\label{eq:stieltjes_function_derivative}
    f^{(k)}(z) = (-1)^k k! \int_0^\infty \frac{\dmu}{(z+t)^{k+1}}.
\end{equation}
and when $f(z) = zg(z)$, where $g$ is a Stieltjes function~\eqref{eq:stieltjes_function}, its derivatives are given by
\begin{equation}\label{eq:derivative_g}
    f^{(k)}(z) = (-1)^{k+1}k! \int_0^\infty \frac{t\dmu}{(z+t)^{k+1}}.
\end{equation}
From these equations, one can easily deduce that both of these function classes have strictly completely monotonic derivatives, so that Theorem~\ref{the:bound_hr} is applicable to them. The following result shows that it is also possible to obtain an \emph{upper} bound on the level-2 condition number.

\begin{theorem}\label{the:lvl2_stieltjes_bound}
Let $f$ be a Stieltjes function or a function of the form $f(z) = zg(z)$, where $g$ is a Stieltjes function and let $A$ be Hermitian positive definite with smallest eigenvalue $\lmin$. Then
\begin{equation}\label{eq:lvl2_stieltjes_bound}
\condtwoabs(f, A) \leq |f^{\prime\prime}(\lmin)|.
\end{equation}
\end{theorem}
\begin{proof}
We begin with the case that $f$ is a Stieltjes function. Then, according to Corollary~\ref{cor:stieltjes}, its second-order Fr\'echet derivative is given by
\begin{equation}\label{eq:stieltjes_second_order_Fd}
L_f^{(2)}(A, E, Z) = \int_0^\infty (A+tI)^{-1}E(A+tI)^{-1}Z(A+tI)^{-1} + (A+tI)^{-1}Z(A+tI)^{-1}E(A+tI)^{-1} \dmu.
\end{equation}
Taking the Frobenius norm in~\eqref{eq:stieltjes_second_order_Fd} and inserting into~\eqref{eq:upper_bound_level2_condition_number} gives
\begin{eqnarray}
    \condtwoabs(f, A) &\leq& \max\limits_{\|Z\|_F=1}\ \max\limits_{\|E\|_F=1} \int_0^\infty \|(A+tI)^{-1}E(A+tI)^{-1}Z(A+tI)^{-1}\|_F \nonumber \\
                    & & \qquad\qquad\qquad\qquad+ \|(A+tI)^{-1}Z(A+tI)^{-1}E(A+tI)^{-1}\|_F \dmu.\label{eq:upper_bound_level2_condition_number_integral}
%                    &\leq& 2\max\limits_{\|Z\|_F=1}\ \max\limits_{\|E\|_F=1} \int_0^\infty \|(A+tI)^{-1}\|_2\|E\|_F\|(A+tI)^{-1}\|_2\|Z\|_F\|(A+tI)^{-1}\|_2 \nonumber \\
%                    &=& 2\int_0^\infty \|(A+tI)^{-1}\|_2^3.
\end{eqnarray}
As the Frobenius norm is unitarily invariant, the inequality $\|BCD\|_F \leq \|B\|_2\|C\|_F\|D\|_2$ holds for any three matrices $B, C, D$ of compatible sizes; see~\cite[Corollary~3.5.10]{HornJohnson1991}. Applying this inequality in~\eqref{eq:upper_bound_level2_condition_number_integral} and noting that $\|B\|_2 \leq \|B\|_F$ for any matrix $B$, we arrive at
\begin{equation}
\condtwoabs(f, A) \leq 2\int_0^\infty \|(A+tI)^{-1}\|_2^3 \dmu. \label{eq:upper_bound_level2_condition_number_integral2}
\end{equation}
As $A$ is Hermitian positive definite, we have $\|(A+tI)^{-1}\|^3 = \frac{1}{(\lmin+t)^3}$. Thus,~\eqref{eq:upper_bound_level2_condition_number_integral2} reduces to
\begin{equation*}\label{eq:bound_using_derivative}
    \condtwoabs(f, A) \leq 2\int_0^\infty \frac{1}{(\lmin+t)^3} \dmu = f^{\prime\prime}(\lmin),
\end{equation*}
where we have used~\eqref{eq:stieltjes_function_derivative} with $k=2$ for the second equality. This proves~\eqref{eq:lvl2_stieltjes_bound} for the case that $f$ is a Stieltjes function. When $f(z) = zg(z)$ with $g$ a Stieltjes function, following the same steps as above after using Theorem~\ref{the:stieltjes_times_z} to represent the second-order \Fd\ yields
\begin{equation*}
\condtwoabs(f, A) \leq 2\int_0^\infty t\|(A+tI)^{-1}\|_2^3 \dmu. \label{eq:upper_bound_level2_condition_number_integral2_zgz}
\end{equation*}
From~\eqref{eq:derivative_g}, we then find
\begin{equation*}
\condtwoabs(f, A) \leq 2\int_0^\infty \frac{t}{(\lmin+t)^3} \dmu = |f^{\prime\prime}(\lmin)|, \label{eq:upper_bound_level2_condition_number_integral3_zgz}
\end{equation*}
which concludes the proof of the theorem.\hfill\qed
\end{proof}

\begin{corollary}\label{cor:lvl2_exact}
Let the assumptions of Theorem~\ref{the:lvl2_stieltjes_bound} hold and assume further that $\lmin$ is a simple eigenvalue of $A$. Then
\begin{equation}\label{eq:lvl2_stieltjes_exact}
\condtwoabs(f, A) = |f^{\prime\prime}(\lmin)|.
\end{equation}
\end{corollary}
\begin{proof}
When $f$ is a Stieltjes function, $f^\prime$ is negative and strictly monotonically increasing on $(0,\infty)$. Therefore, the maximum in~\eqref{eq:condabs1} is attained for $\lmin$ and if $\lmin$ is a simple eigenvalue of $A$, we get from~\eqref{eq:condabs2} in Theorem~\ref{the:bound_hr} that $\condtwoabs(f,A) \geq |f^{\prime\prime}(\lmin)|$. From Theorem~\ref{the:lvl2_stieltjes_bound}, we further have that $\condtwoabs(f,A) \leq |f^{\prime\prime}(\lmin)|$. Both inequalities together imply~\eqref{eq:lvl2_stieltjes_exact}. The case $f(z) = zg(z)$ follows analogously, noting that in this case $f^\prime$ is positive and strictly monotonically decreasing on $(0,\infty)$, so that the maximum in\eqref{eq:condabs1} is again attained for $\lmin$.\hfill\qed
\end{proof}

\begin{remark}\label{rem:cond_inverse}
The function $f(z) = z^{-1}$ is a Stieltjes function, so for a matrix that fulfills the assumptions of Corollary~\ref{cor:lvl2_exact}, we have that 
$$\condtwoabs(z^{-1}, A) = \frac{2}{\lmin^3} = 2\|A^{-1}\|_2^3,$$
which exactly agrees with the formula given for the level-2 condition number of the inverse in~\cite[Eq.~(5.13)]{HighamRelton2014}. Note, however, that the formula for the inverse from~\cite{HighamRelton2014} holds for \emph{any} nonsingular $A$, not necessarily Hermitian or positive definite.
\end{remark}

\begin{experiment}\label{ex:lvl2_conditionnumber}
Corollary~8 gives an exact expression for the level-two condition number in cases where previously only bounds where available. While our exact expression in fact agrees with the lower bound that was previously available, it is instructive to compare it to the upper bounds~\eqref{eq:upper_bound_level2_condition_number}: Whenever it is of interest to use the condition number in order to obtain accuracy guarantees for a computed result, upper bounds are of much more use than lower bounds. In this experiment, we compare~\eqref{eq:upper_bound_level2_condition_number} and~\eqref{eq:lvl2_stieltjes_exact} for matrices from the \texttt{anymatrix} collection~\cite{Anymatrix}\footnote{available at \url{https://github.com/mmikaitis/anymatrix}}. Specifically, we select all matrices which \dots
\begin{itemize}
\item[a)] \dots are Hermitian positive definite (where we also exclude matrices that are numerically singular, i.e., have a positive smallest eigenvalue in the order of the unit roundoff),
\item[b)] \dots have a smallest eigenvalue with multiplicity one, \dots
\item[c)] have a size of $n \leq 10$ or are scalable in size.
\end{itemize}
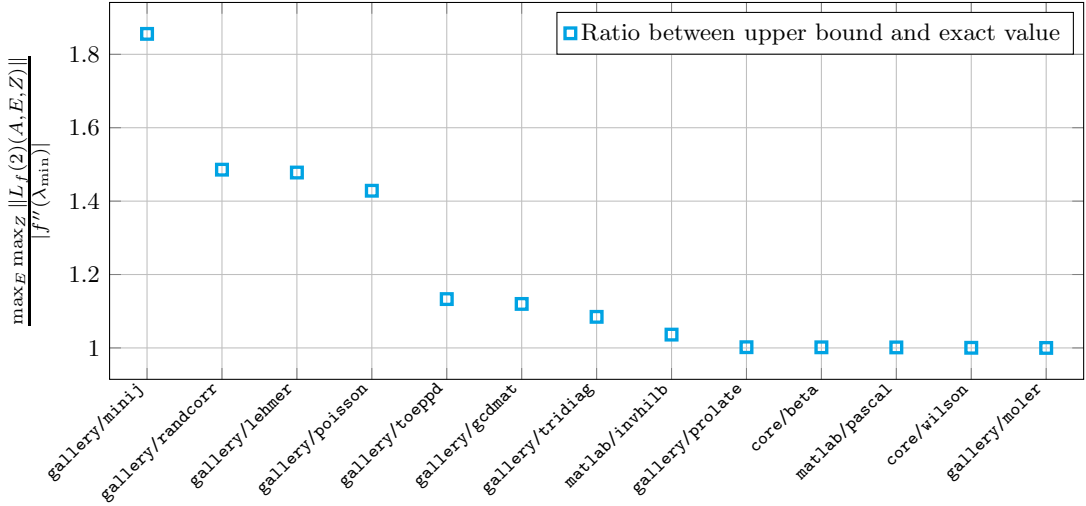
\begin{figure}
\centering
\tikzsetnextfilename{ratio}
\pgfplotsset{height=0.4\linewidth,width=0.875\linewidth,compat=1.10,every axis/.append style={legend style={/tikz/every even column/.append style={column sep=6pt}}}}
\pgfplotsset{every tick label/.append style={font=\small}}
%\tikzexternaldisable

\noindent%
\begin{tikzpicture}[scale=1]%
    \begin{axis}[legend style={at={(.99,0.975)}}, 
   	anchor=north east, legend columns=1,cycle list name=list_std3, 
   	xmin=0.5, xmax=13.5,grid=major, 
   	xlabel={}, ylabel={$\frac{\max_E \max_Z \|L_f(2)(A,E,Z)\|}{|f^{\prime\prime}(\lmin)|}$}, 
   	title={Comparison of exact level-two condition number and upper bound},
   	xticklabels={\texttt{\scriptsize gallery/minij}, \texttt{\scriptsize gallery/randcorr}, \texttt{\scriptsize gallery/lehmer}, \texttt{\scriptsize gallery/poisson}, \texttt{\scriptsize gallery/toeppd}, \texttt{\scriptsize gallery/gcdmat}, \texttt{\scriptsize gallery/tridiag}, \texttt{\scriptsize matlab/invhilb}, \texttt{\scriptsize gallery/prolate}, \texttt{\scriptsize core/beta}, \texttt{\scriptsize matlab/pascal}, \texttt{\scriptsize core/wilson}, \texttt{\scriptsize gallery/moler}},xtick={1,...,13}, x tick label style={rotate=45,anchor=east}]

\addplot+[only marks, mark=square] table [x ={ind},y ={ratio}] {figs/dat/test_condnum.dat};\addlegendentry{\small Ratio between upper bound and exact value}

\end{axis}
\end{tikzpicture}
\caption{Ratio between upper bound~\eqref{eq:upper_bound_level2_condition_number} and exact level-two condition number~\eqref{eq:lvl2_stieltjes_exact} for $f(z) = z^{-1/2}$ and various Hermitian positive definite matrices from the \texttt{anymatrix} collection. }\label{fig:experiment_ratio}
\end{figure}
While conditions a) and b) are clearly necessary for Theorem~\ref{the:lvl2_stieltjes_bound} and Corollary~\ref{cor:lvl2_exact} to hold, condition c) is included due to the very high cost of $\mathcal{O}(n^7)$ required for computing the Kronecker form of the second-order \Fd{} using~\cite[Algorithm~4.2]{HighamRelton2014}. In Figure~\ref{fig:experiment_ratio}, we plot the ratio between the upper bound~\eqref{eq:upper_bound_level2_condition_number} and the exact value~\eqref{eq:lvl2_stieltjes_exact} for all matrices fulfilling a)-c). We can observe that our new result actually does not improve over the upper bound by too much in most cases: The largest improvement is by a factor of about $1.9$, while for many matrices, the two values more or less agree. However, we want to stress that the main benefit of our result is that it gives a very concise, closed form of the level-two condition number and that it can be computed very efficiently. It only requires finding the smallest eigenvalue of $A$, while naively evaluating the upper bound~\eqref{eq:upper_bound_level2_condition_number} has complexity $\mathcal{O}(n^7)$ and requires $\mathcal{O}(n^6)$ storage.
\end{experiment}

\section{Computational methods for higher-order \Fd{}s}\label{sec:algorithms}
This section is devoted to computational methods for approximating $L_f^{(k)}(A, E_1,\dots,E_k)$. In particular, we discuss how the integral representations that we have derived in Section~\ref{sec:integral_representation} can be used in the context of numerical quadrature methods. Before that, we briefly recap some methods proposed in the literature for approximating the higher-order \Fd{}.

\subsection{Established methods for approximating the higher-order \Fd{}}\label{subsec:estalished_methods}

In~\cite{HighamRelton2014}, Higham and Relton prove the following theorem which relates the higher-order \Fd\ to the result of evaluating $f$ at a large, Kronecker-structured matrix.

\begin{theorem}[Theorem 3.5 in~\cite{HighamRelton2014}]\label{the:X_representation}
Let the largest Jordan block of $A\in \Cnn$ be of size $p \times p$ and let $\spec(A) \subseteq \calD$ where $\calD$ is an open subset of the complex plane. If $f: \calD \longrightarrow \C$ is $2^kp - 1$ times continuously differentiable on $\calD$, then 
\begin{itemize}
    \item[(i)] the $k$th-order \Fd\ $L_f^{(k)}(A, \cdot, \dots, \cdot)$ exists and is continuous in $A, E_1, \dots, E_k$,
    \item[(ii)] $L_f^{(k)}(A, E_1, \dots, E_k)$ is equal to the upper right $n \times n$ block of $f(X_k)$, where 
    \begin{equation}\label{eq:X}
        X_k = \begin{cases}
        A & k = 0 \\
        I_{2 \times 2} \otimes X_{k-1} + \left[\begin{array}{cc} 0 & 1 \\ 0 & 0\end{array}\right] \otimes I_{2^{k-1} \times 2^{k-1}} \otimes E_k & k \geq 1,
        \end{cases}
    \end{equation}
    and $\otimes$ denotes the Kronecker product.
\end{itemize}
\end{theorem}

Based on Theorem~\ref{the:X_representation}, Higham and Relton~\cite[Algorithm 3.6]{HighamRelton2014} propose to compute $f(X_k)$, with $X_k$ defined in~\eqref{eq:X}, and then extract its upper-right $n \times n$ block in order to obtain $L_f^{(k)}(A, E_1, \dots, E_k)$. As the matrix $X_k$ is of size $2^kn \times 2^kn$, this approach quickly becomes prohibitively expensive for growing $n$ and/or $k$. Assuming that an algorithm is available for evaluating $f$ at a matrix argument with complexity $\calO(n^3)$ (as is the case for many important matrix functions), evaluating $f(X_k)$ has complexity $\calO(8^k n^3)$. 

Recently, Al-Mohy and Arslan~\cite{AlMohyArslan2020} have investigated the use of the complex-step approximation for the higher-order \Fd; see also~\cite{Werner2021} for related work by Werner. In particular, they show that if $f$ is an analytic function that takes real values at real arguments,
\begin{equation}\label{eq:complex_step}
L_f^{(k)}(A, E_1, \dots, E_k) = \frac{1}{h} \Im \left(L_f^{(k-1)}(A+ihE_k,E_1,\dots,E_k)\right) + \mathcal{O}(h^2),
\end{equation}
where $\Im(M)$ denotes the (entrywise) imaginary part of a matrix $M$. Thus, defining $X_{k-1}(h)$ according to~\eqref{eq:X}, starting with $X_0 = A + ihE_k$ instead of $X_0 = A$, the upper right $n \times n$ block of $\frac{1}{h} \Im (f(X_{k-1}(h)))$ is an approximation for $L_f^{(k)}(A,E_1,\dots,E_k)$. Using this approach, $f$ needs to be evaluated at a square matrix of size $2^{k-1}n$ instead of $2^kn$, thus reducing the computational complexity by a factor of $8$. This comes at the cost of introducing complex arithmetic even if $A$ is real, though, which increases complexity by a factor of roughly $4$. Thus, for real $A$, this approach reduces the cost to about 50\% compared to the algorithm of Higham and Relton.
By using a mixed derivative approach, Al-Mohy and Arslan are able to further reduce the cost by a factor of $4$, i.e., to about 12.5\% of the cost of~\cite[Algorithm 3.6]{HighamRelton2014}. Using this approach requires choosing a second step size parameter though, and in~\cite[Experiment 2]{AlMohyArslan2020}, it is demonstrated that the quality of the approximation is quite sensitive to the choice of this second parameter and that substantial subtractive cancellation might occur, leading to numerical instabilities; the use of this approximation is thus not advocated in general.

A further noteworthy observation in~\cite{AlMohyArslan2020} is that it is actually unnecessary to compute the whole matrix $f(X_k)$ in all of the mentioned algorithms. Instead, defining the block-vector
\begin{equation*}\label{eq:B}
B = [0, \dots, 0, I]^T \in \R^{2^kn \times n},
\end{equation*}
where $0$ and $I$ denote the zero and identity matrix of size $n \times n$, respectively, it suffices to compute $f(X_k)B$ and extract its top $n \times n$ block. In cases were an efficient routine for approximating the action of $f(X_k)$ on a (block) vector is available (as is the case, e.g., for the exponential, using the algorithm by Al-Mohy and Higham~\cite[Algorithm 3.2]{AlMohyHigham2011}), this can lead to huge speed-ups for the computation of the higher-order \Fd---at least for large derivative order $k$---as demonstrated in~\cite[Experiment~1]{AlMohyArslan2020}.

\subsection{Approximating the higher-order \Fd\ by numerical quadrature}\label{subsec:quadrature_algorithm}
Using quadrature rules for approximating matrix functions or their action on vectors is a well-established and widely used technique; see, e.g.,~\cite{Cardoso2012, FrommerLundSzyld2017, Hale2008, TrefethenWeidemanSchmelzer2006, WeidemanTrefethen2007, FrommerGuettelSchweitzer2014a, Schweitzer2016} and the references therein. Thus, in light of the integral representations found in Section~\ref{sec:integral_representation}, it is a natural approach to also try to approximate the higher-order \Fd\ by quadrature rules. In the following, we briefly outline this approach and highlight a few computational aspects. In order to avoid unnecessary repetitions, we explain most concepts only for the Cauchy integral formula~\eqref{eq:cauchy_integral_fA}, as all formulas hold for Stieltjes and related representations with obvious modifications.

Applying an $m$-node quadrature rule with nodes $\zeta_j, j = 1,\dots,m$ and weights $w_j, j = 1,\dots, m$ to approximate the integral~\eqref{eq:integral_representation} yields an approximation
\begin{equation}\label{eq:frechet_quadrature}
L_f^{(k)}(A, E_1, \dots, E_k) \approx \widetilde{L}_m(A,E_1,\dots,E_k) := \sum\limits_{j = 1}^m\sum\limits_{\pi\in S_k} w_j f(\zeta_j) M_\pi(\zeta_j; A, E_1, \dots, E_k),
\end{equation}
 where the matrices $M_\pi(\zeta_j; A, E_1, \dots, E_k)$ are defined according to~\eqref{eq:integrand_M}. In contrast to the approaches of \cite{HighamRelton2014,AlMohyArslan2020}, only matrices of size $n \times n$ appear in~\eqref{eq:frechet_quadrature}, irrespective of the order $k$ of the \Fd. However, the number of terms in the inner sum is $|S_k| = k!$, which grows rapidly with $k$, suggesting that the use of~\eqref{eq:frechet_quadrature} might be particularly attractive for moderate values of $k$. % Other factors influencing the computational efficiency of the approximation~\eqref{eq:frechet_quadrature} are, e.g., the number of quadrature nodes that are necessary for achieving the desired accuracy and the nonzero pattern of $A$ and $E_1, \dots, E_k$, determining how efficient matrix factorizations and products can be carried out.

To obtain an actual algorithm for a specific function $f$, it is necessary to choose the nodes $\zeta_j$ and weights $w_j$ of the quadrature rule. Fortunately, we can use the same quadrature rules in our setting as one would also use in algorithms for approximating $f(A)$ or $f(A)\vb$~\cite{Hale2008,TrefethenWeidemanSchmelzer2006,FrommerGuettelSchweitzer2014a}, as the following result guarantees. For brevity, we state it only for functions represented by the Cauchy integral formula and note that it holds in the same way for, e.g., Stieltjes functions. 
\begin{proposition}\label{prop:convergence_quadrature}
Let $f$ be representable by the Cauchy integral formula~\eqref{eq:cauchy_integral_fA}, let $\zeta_j, w_j, j = 1,\dots,m$ be the nodes and weights of an $m$-point quadrature rule for $f(A)$ and assume that for $\spec(A) \subseteq \Omega$, this quadrature rule converges as
\begin{equation*}\label{eq:geometric_convergence_assumption}
\left\|f(A) - \frac{1}{2\pi i} \sum_{j=1}^m w_j f(\zeta_j) (\zeta_j I-A)^{-1}\right\| = \mathcal{O}(q_{f, \Omega}(m)),
\end{equation*}
where $q_{f, \Omega}(m)$ is such that $\lim_{m\rightarrow\infty} q_{f, \Omega}(m) = 0$. Then the quadrature approximation $\widetilde{L}_m(A,E_1,\dots,E_k)$ defined in~\eqref{eq:frechet_quadrature} also fulfills
\begin{equation}\label{eq:geometric_convergence_result}
\|L_f^{(k)}(A, E_1,\dots,E_k) - \widetilde{L}_m(A,E_1,\dots,E_k)\| = \mathcal{O}(q_{f, \Omega}(m))
\end{equation}
when $\spec(A) \subseteq \Omega$.
\end{proposition}
\begin{proof}
Let $\spec(A) \subseteq \Omega$ and let $X_k$ be the matrix defined in~\eqref{eq:X}. The matrix $X_k$ is block upper triangular with all diagonal blocks equal to $A$, so that $\spec(X_k) \subseteq \Omega$ and therefore, by the assumption~\eqref{eq:geometric_convergence_assumption}, we have
\begin{equation}\label{eq:geometric_convergence_proof}
\left\|f(X_k) - \frac{1}{2\pi i} \sum_{j=1}^m w_j f(\zeta_j) (\zeta_j I-X_k)^{-1}\right\| = \mathcal{O}(q_{f, \Omega}(m)).
\end{equation}
Clearly, as~\eqref{eq:geometric_convergence_proof} holds for the complete matrix $f(X_k)$, it holds in particular for any sub-block. According to Theorem~\ref{the:X_representation}, the upper right $n \times n$ block of $f(X_k)$ is exactly equal to $L_f(A, E_1,\dots, E_k)$. In the same way, the upper right $n \times n$ block of $(\zeta_j I-X_k)^{-1}$ equals $L_{(\zeta_j-z)^{-1}}(A,E_1,\dots,E_k)$. Thus, from Lemma~\ref{lem:kth_fd_resolvent}, we  can conclude that the upper right block of $\frac{1}{2\pi i} \sum_{j=1}^m w_j f(\zeta_j) (\zeta_j I-X_k)^{-1}$ 
agrees with $\widetilde{L}_m(A,E_1,\dots,E_k)$. Therefore, by restricting~\eqref{eq:geometric_convergence_proof} to the upper right $n \times n$ block, the assertion follows.\hfill\qed
\end{proof}

\begin{remark}\label{rem:quadrature_convergence}
Often, quadrature rules for matrix functions converge geometrically. For example, for $f(z) = \exp(z), \Omega = (-\infty, 0]$ and quadrature nodes and weights on an optimized parabolic contour as considered in~\cite{TrefethenWeidemanSchmelzer2006} yields $q_{f, \Omega}(m) = 2.85^{-m}$, while for $f(z) = \sqrt{z}, \Omega = [\alpha, \beta] \subseteq (0, \infty)$ one obtains $q_{f, \Omega}(m) = e^{-\pi^2 m / (\log (\beta/\alpha)+3)}$ by using the quadrature rule labeled ``Method 1'' in~\cite{Hale2008}.
\end{remark}

In order to not make this section too unfocused, a discussion of important matrix function quadrature rules (and in particular the rules used in our numerical experiments) can be found in~\ref{sec:appendix_quadrature}.

\subsection{Efficiently implementing the quadrature approximation}\label{subsec:efficient_implementation}
How to evaluate~\eqref{eq:frechet_quadrature} most efficiently depends on the properties of the matrix $A$ and the direction terms $E_1, \dots, E_k$ (and of course also on other factors like the hardware architecture, a topic that is well beyond the scope of this manuscript). In cases where maximum efficiency is important, the implementation can and should therefore be specifically optimized with regard to the circumstances. In the following, we exemplarily cover two important situations.
\paragraph{When $A$ and $E_1,\dots,E_k$ are dense and unstructured}
The first situation we consider is that both $A$ and $E_1,\dots,E_k$ do not have any structure that can be exploited, i.e., we assume that they are dense, of (almost) full rank and possibly non-Hermitian. Still, by exploiting the structure of the terms $M_\pi(\zeta_j; A, E_1,\dots,E_k)$ occurring in the quadrature rule, we can save computational effort over a naive implementation, in particular concerning the repeated application of the resolvents $(\zeta_j I - A)^{-1}$. To do so, for each node $\zeta_j$, we compute an LU decomposition\added{ (with partial pivoting)}
\begin{equation*}\label{eq:LU}
\added{P_j(}\zeta_j I - A\added{)} = L_jU_j,
\end{equation*}
upfront and then form the auxiliary matrices
\begin{equation*}\label{eq:E_tilde}
\widetilde{E}_i^{(j)} = L_j^{-1} \added{P_j}E_i U_j^{-1}, \quad i = 1,\dots, k.    
\end{equation*}
We then have
\begin{eqnarray*}
M_\pi(\zeta_j; A, E_1,\dots,E_k) &=& (\zeta_j I - A)^{-1}E_{\pi(1)}(\zeta_j I - A)^{-1}E_{\pi(2)}(\zeta_j I - A)^{-1}\cdots E_{\pi(k)}(\zeta_j I - A)^{-1} \nonumber\\
&=& U_j^{-1}L_j^{-1}\added{P_j}E_{\pi(1)}U_j^{-1}L_j^{-1}\added{P_j}E_{\pi(2)}\cdots L_j^{-1}\added{P_j}E_{\pi(k)}U_j^{-1}L_j^{-1}\added{P_j} \nonumber\\
&=& U_j^{-1} \widetilde{E}_{\pi(1)}^{(j)}\widetilde{E}_{\pi(2)}^{(j)} \cdots\widetilde{E}_{\pi(k)}^{(j)}L_j^{-1}\added{P_j}. \label{eq:product_E_tilde}
\end{eqnarray*}
This way, the number of matrix products that need to be computed for each of the terms in the inner sum in~\eqref{eq:frechet_quadrature} is reduced and the effort for computing the action of the resolvent is only performed once for each $E_i$. Using this approach, the evaluation of~\eqref{eq:frechet_quadrature} with $m$ quadrature nodes has a computational complexity of $\mathcal{O}(mn^3k!)$.

\begin{remark}
Of course, in the approach outlined above, intermediate matrix products $\widetilde{E}_{\pi(1)}^{(j)} \cdot \widetilde{E}_{\pi(2)}^{(j)} \cdot{}\dots{}\cdot \widetilde{E}_{\pi(\ell)}^{(j)}$ of length $\ell < k$ are required several times in the computation of all terms in the inner sum in~\eqref{eq:frechet_quadrature}. Depending on the available memory and the values of $n$ and $k$, one can compute all of these intermediate products just once, store them, and reuse them in a dynamic programming fashion. However, as the number of intermediate terms grows rapidly with $k$, this approach can only be used for moderate values of $k$ and we do not pursue this further here. One could also think of only partially using this approach, e.g., precomputing and reusing only the $k(k-1)$ products of length $\ell = 2$.
\end{remark}

\paragraph{Rank-one direction terms}
Another very important special case is that the direction matrices $E_1, \dots, E_k$ have rank one, or, even more specifically, that the direction terms are outer products of canonical unit vectors,  $E_i = \ve_{\alpha_i}\ve_{\beta_i}^T$ for some indices $\alpha_i, \beta_i, i = 1,\dots,k$. This situation occurs, e.g., when computing the Kronecker form of the $k$th \Fd; see~\cite[Algorithm 4.2]{HighamRelton2014}. Let us denote
\begin{equation}\label{eq:vectors_rank1}
\va_i = (\zeta_j I - A)^{-1}\ve_{\alpha_i} \quad\text{ and }\quad \vb_i = (\zeta_j I - A)^{-H}\ve_{\beta_i}, i = 1,\dots, k.
\end{equation}
Then, by writing
\begin{eqnarray}
M_\pi(\zeta; A, E_1,\dots,E_k)  &=& (\zeta I - A)^{-1}\ve_{\alpha_{\pi(1)}}\ve_{\beta_{\pi(1)}}^T(\zeta I - A)^{-1}\ve_{\alpha_{\pi(2)}}\ve_{\beta_{\pi(2)}}^T(\zeta I - A)^{-1}\cdots \ve_{\alpha_{\pi(k)}}\ve_{\beta_{\pi(k)}}^T(\zeta I - A)^{-1} \nonumber\\
&=& \left((\vb_{\pi(1)}^H\ve_{\pi(2)}) \cdot (\vb_{\pi(2)}^H\ve_{\pi(3)}) \cdot{}\dots{}\cdot (\vb_{\pi(k-1)}^H\ve_{\pi(k)}) \right) \cdot \va_{\pi(1)}\vb_{\pi(k)}^H,\label{eq:inner_and_outer_products}
\end{eqnarray}
one sees that $M_\pi(\zeta; A, E_1,\dots,E_k)$ is also a rank-one matrix. Computing the product in the form~\eqref{eq:product_E_tilde} would not fully exploit this structure. Instead, it is preferable to precompute the vectors $\va_1,\dots,\va_k$ and $\vb_1, \dots, \vb_k$ from~\eqref{eq:vectors_rank1} and then use those to form $M_\pi(\zeta_j; A, E_1,\dots,E_k)$. Having already solved the linear systems, the additional work needed to form a single matrix $M_\pi(\zeta_j; A, E_1,\dots,E_k)$ then amounts to computing $k-1$ inner products and one outer product, yielding a computational complexity of roughly $n^2 + 2(k-1)n$ operations. The overall complexity (ignoring lower order terms) of the resulting method is then given by
\begin{equation}\label{eq:complexity_rank_one}
    m\cdot (2k\cdot\texttt{solve}(n) + n^2k!),
\end{equation}
where $\texttt{solve}(n)$ denotes the complexity of solving a linear system with (a shifted version of) $A$ or $A^H$. If $\texttt{solve}(n) = \mathcal{O}(n^2)$, the overall algorithm then scales as $n^2$ for fixed $k$ and $m$. A typical situation for which this is the case is if $A$ is banded (with bandwidth $\gamma$ independent of $n$), so that the linear systems~\eqref{eq:vectors_rank1} can be solved with complexity $\mathcal{O}(\gamma^2 n)$ by a sparse direct solver.

\begin{figure}
\tikzsetnextfilename{experiment_sparseAsparseE}
\pgfplotsset{height=0.4\linewidth,width=0.46\linewidth,compat=1.10,every axis/.append style={legend style={/tikz/every even column/.append style={column sep=6pt}}}}
\pgfplotsset{every tick label/.append style={font=\small}}
%\tikzexternaldisable

\noindent%
\begin{tikzpicture}[scale=1]%
    \begin{semilogyaxis}[legend style={at={(0.275,0.975)}}, 
   	anchor=north east, legend columns=1,cycle list name=list_std, 
   	xmin=1.8, xmax=9.2,grid=major, 
   	xlabel={\small $k$}, ylabel={\small time [s]}, 
   	title={\small $n = 50$ fixed}]

\addplot+[]
table [x ={k},y ={HR}] {figs/dat/experiment_tridiagonal_rankone_k.dat};\addlegendentry{\small  \texttt{HR}}
\addplot+[]
table [x ={k},y ={AA}] {figs/dat/experiment_tridiagonal_rankone_k.dat};\addlegendentry{\small  \texttt{AA}}
\addplot+[]
table [x ={k},y ={Q}] {figs/dat/experiment_tridiagonal_rankone_k.dat};\addlegendentry{\small  \texttt{Q}}

    \end{semilogyaxis}
\end{tikzpicture}
\hspace{.3cm}
\begin{tikzpicture}[scale=1]%
    \begin{semilogyaxis}[legend style={at={(0.275,0.975)}}, 
   	anchor=north east, legend columns=1,cycle list name=list_std, 
   	xmin=40, xmax=360,grid=major, 
   	xlabel={\small $n$}, ylabel={\small time [s]}, 
   	title={\small $k = 4$ fixed}]

\addplot+[]
table [x ={n},y ={HR}] {figs/dat/experiment_tridiagonal_rankone_n.dat};\addlegendentry{\small  \texttt{HR}}
\addplot+[]
table [x ={n},y ={AA}] {figs/dat/experiment_tridiagonal_rankone_n.dat};\addlegendentry{\small  \texttt{AA}}
\addplot+[]
table [x ={n},y ={Q}] {figs/dat/experiment_tridiagonal_rankone_n.dat};\addlegendentry{\small  \texttt{Q}}
\addplot+[dashed, black,mark=none]
table [x ={n},y ={est}] {figs/dat/experiment_tridiagonal_rankone_n.dat};
\addplot+[dotted, black,mark=none]
table [x ={n},y ={est2}] {figs/dat/experiment_tridiagonal_rankone_n.dat};

    \end{semilogyaxis}
\end{tikzpicture}
\cprotect\caption{Results of Experiment~\ref{ex:lesp_sparse}: Comparison of run-time of the different algorithms for $n = 50$ and varying $k$ (left) and for $k = 4$ and varying $n$ (right). In both cases, the matrix $A \in \Rnn$ is the tridiagonal matrix generated by the command \verb|A = gallery('lesp',n)| and the matrices $E_1, \dots, E_k$ are outer products of canonical unit vectors. The dotted and dashed line in the right plot indicate quadratic and cubic scaling with respect to $n$, respectively.}\label{fig:experiment_sparseAsparseE}
\end{figure}
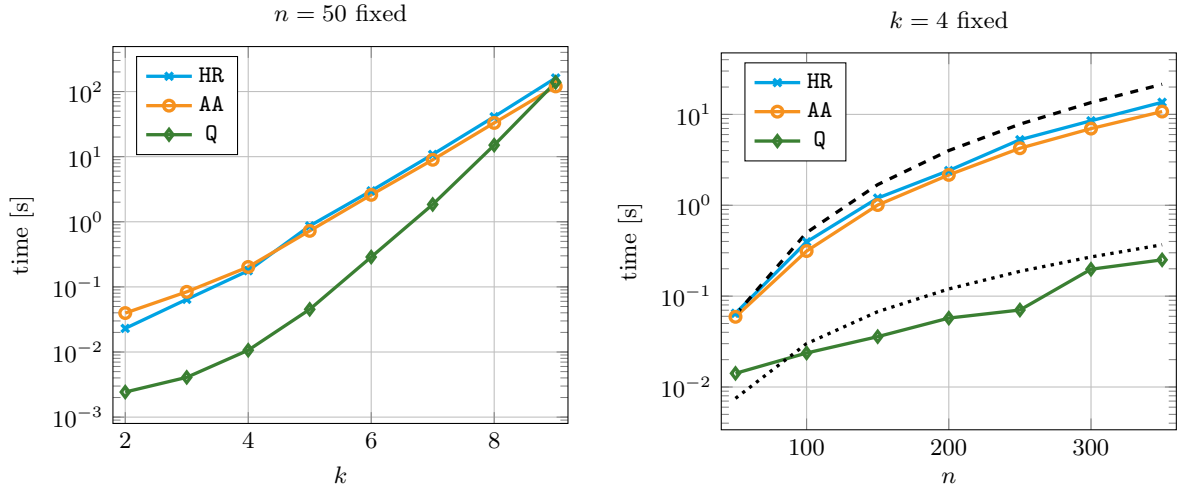

\section{Numerical experiments}\label{sec:numerical_experiments}
In this section, we compare the quadrature-based algorithm proposed in Section~\ref{subsec:quadrature_algorithm} to the algorithms from~\cite{HighamRelton2014, AlMohyArslan2020} which were briefly discussed in Section~\ref{subsec:estalished_methods}. All experiments are carried out in MATLAB R2021b on a PC with an AMD Ryzen 7 3700X 8-core CPU and 32 GB RAM. To make run time comparisons more reliable, we use the option \texttt{maxNumCompThreads = 1} so that MATLAB is limited to a single computational thread. 
The algorithms that we compare are listed in the following. 
\begin{itemize}
    \item \texttt{HR}: The original algorithm by Higham and Relton from~\cite{HighamRelton2014}, based on Theorem~\ref{the:X_representation}(ii).
    \item \texttt{AA}: The complex-step approximation algorithm by Arslan and Al-Mohy~\cite{AlMohyArslan2020}, based on~\eqref{eq:complex_step}.
    \item \texttt{Q}: The approximation based on the numerical quadrature formula~\eqref{eq:frechet_quadrature}.
\end{itemize}
In our first two experiments, we choose $f(z) = \exp(z)$, which allows us to use the efficient method \texttt{expmv}~\cite{AlMohyHigham2011}\footnote{available at \url{https://github.com/higham/expmv}} to compute the action of the matrix exponential on a (block) vector in the algorithms \texttt{HR} and \texttt{AA} that we compare against. In the quadrature based algorithm \texttt{Q} we choose $\Gamma$ as the parabolic contour~\eqref{eq:parabolic_contour}.

\begin{experiment}\label{ex:lesp_sparse}
In this first experiment, we use a similar setup to that of~\cite[Experiment~1]{AlMohyArslan2020}: We generate $A \in \Rnn$ by the MATLAB command \verb|A = gallery('lesp',n)|, which results in a tridiagonal matrix with smoothly distributed eigenvalues in the interval $[-(3.5+2n), -4.5]$. The matrices $E_i$ are chosen each as outer product of two canonical unit vectors, i.e., $E_i = \ve_{\alpha_i}\ve_{\beta_i}^T$ for random $\alpha_i,\beta_i \in \{1,\dots,n\}$. We are therefore in the situation outlined in the paragraph ``Rank-one direction terms'' of Section~\ref{subsec:efficient_implementation} and accordingly evaluate $M_\pi(\zeta; A, E_1,\dots,E_k)$ using~\eqref{eq:vectors_rank1}--\eqref{eq:inner_and_outer_products}.

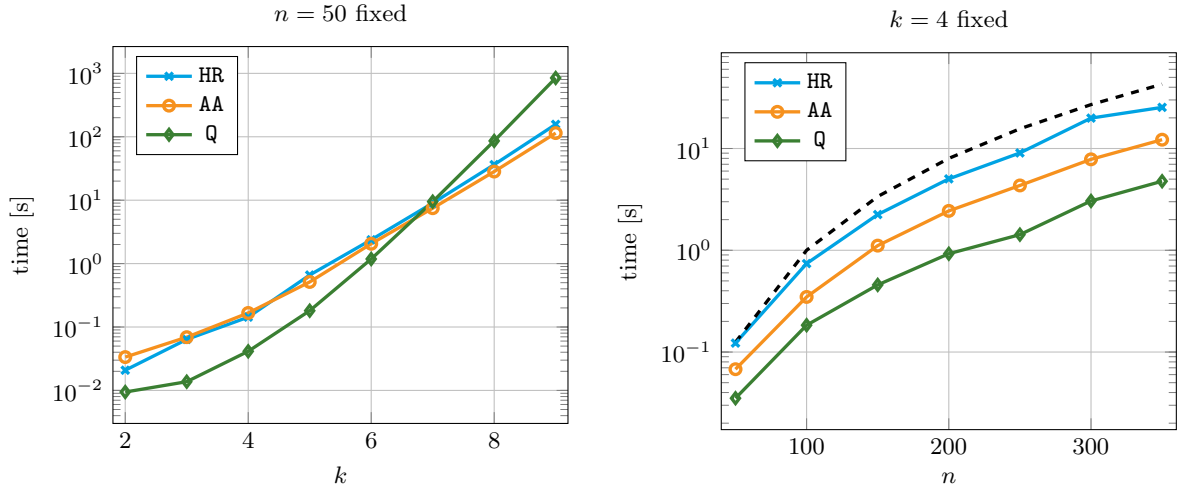
\begin{figure}
\tikzsetnextfilename{experiment_fullAfullE}
\pgfplotsset{height=0.4\linewidth,width=0.46\linewidth,compat=1.10,every axis/.append style={legend style={/tikz/every even column/.append style={column sep=6pt}}}}
\pgfplotsset{every tick label/.append style={font=\small}}
%\tikzexternaldisable

\noindent%
\begin{tikzpicture}[scale=1]%
    \begin{semilogyaxis}[legend style={at={(0.275,0.975)}}, 
   	anchor=north east, legend columns=1,cycle list name=list_std, 
   	xmin=1.8, xmax=9.2,grid=major,	xlabel={\small $k$}, ylabel={\small time [s]}, 
   	title={\small $n = 50$ fixed}]

\addplot+[]
table [x ={k},y ={HR}] {figs/dat/experiment_unstructured_k.dat};\addlegendentry{\small \texttt{HR}}
\addplot+[]
table [x ={k},y ={AA}] {figs/dat/experiment_unstructured_k.dat};\addlegendentry{\small  \texttt{AA}}
\addplot+[]
table [x ={k},y ={Q}] {figs/dat/experiment_unstructured_k.dat};\addlegendentry{\small  \texttt{Q}}

    \end{semilogyaxis}
\end{tikzpicture}
\hspace{.3cm}
\begin{tikzpicture}[scale=1]%
    \begin{semilogyaxis}[legend style={at={(0.275,0.975)}}, 
   	anchor=north east, legend columns=1,cycle list name=list_std, 
   	xmin=40, xmax=360,grid=major, 
   	xlabel={\small $n$}, ylabel={\small time [s]}, 
   	title={\small $k = 4$ fixed}]

\addplot+[]
table [x ={n},y ={HR}] {figs/dat/experiment_unstructured_n.dat};\addlegendentry{\small  \texttt{HR}}
\addplot+[]
table [x ={n},y ={AA}] {figs/dat/experiment_unstructured_n.dat};\addlegendentry{\small  \texttt{AA}}
\addplot+[]
table [x ={n},y ={Q}] {figs/dat/experiment_unstructured_n.dat};\addlegendentry{\small  \texttt{Q}}
\addplot+[dashed, black,mark=none]
table [x ={n},y ={est}] {figs/dat/experiment_unstructured_n.dat};

    \end{semilogyaxis}
\end{tikzpicture}
\caption{Comparison of run-time of the different algorithms for $n = 50$ and varying $k$ (left) and for $k = 4$ and varying $n$ (right). In both cases, the matrix $A \in \Rnn$ is a dense matrix having the same eigenvalues as the matrix used in the previous experiment and the matrices $E_1, \dots, E_k$ have normally distributed random entries. The dashed line in the right plot indicates cubic scaling with respect to $n$. }\label{fig:experiment_fullAfullE}
\end{figure}

We perform two series of experiments. First we fix the matrix size $n = 50$ and compute the $k$th order \Fd\ for $k = 2, \dots, 9$, then we fix $k = 4$ and compute the fourth-order \Fd\ for matrices with sizes ranging from $n = 50$ to $n = 350$. We always use $m = 40$ quadrature nodes in method \texttt{Q}, which is sufficient to reach a relative approximation error in the order of machine precision for all considered problem sizes. When fixing $n$ and increasing $k$, we use \texttt{expmv} in algorithms \texttt{HR} and \texttt{AA}, while for moderate $k$ and increasing $n$, we compute the full matrix function $\exp(X_k)$ using the built-in MATLAB function \texttt{expm}, which turns out to be more efficient (see also~\cite{AlMohyArslan2020}, where the same observation was made).

Results for varying $k$ are reported on the left-hand side of Figure~\ref{fig:experiment_sparseAsparseE}. We observe that for $k= 2,\dots,8$, \texttt{Q} shows the best performance, but also that it shows the worst scaling behavior, clearly illustrating the influence of the factorial term in the complexity bound, making it the worst performing method for $k = 9$. The algorithms \texttt{HR} and \texttt{AA} show comparable performance, with \texttt{AA} being slightly faster for larger values of $k$. 

The results for varying $n$ and fixed $k$ are depicted on the right-hand side of Figure~\ref{fig:experiment_sparseAsparseE}. Here, \texttt{Q} is the fastest algorithm in all cases, showing that it scales best with increasing $n$ when the value of $k$ is moderate. As in the first part of the experiment, \texttt{HR} and \texttt{AA} show comparable performance (again with a slight advantage for \texttt{AA}). Linear systems with $A$ can be solved with complexity $\mathcal{O}(n)$, as $A$ is tridiagonal, so we expect the complexity of \texttt{Q} to scale as $\mathcal{O}(n^2)$ for fixed $k$; cf.~\eqref{eq:complexity_rank_one}, while the complexity of \texttt{HR} and \texttt{AA} scales as $\mathcal{O}(n^3)$. These theoretical considerations are clearly confirmed by the result of the experiment, also illustrated by the dotted and dashed lines indicating quadratic and cubic scaling, respectively. 
\end{experiment}

\begin{experiment}\label{ex:full}
In our second experiment, we test the algorithms in a situation where both $A$ and $E_i, i = 1,\dots,k$ are dense and unstructured. To make the results easily comparable to those of Experiment~\ref{ex:lesp_sparse}, we choose $A$ such that it has the same eigenvalues as the matrix generated by the MATLAB command \verb|A = gallery('lesp',n)|. To so, we apply a random orthogonal similarity transformation to this matrix, which results in a completely unstructured matrix. The matrices $E_i$ are dense matrices with normally distributed random entries, generated by the MATLAB command \texttt{E = randn(n)}.

Apart from the choice of $A$ and $E_i$, the experimental setup is exactly the same as in Experiment~\ref{ex:lesp_sparse}, i.e., we again perform two series of experiments, one for fixed $n$ and one for fixed $k$. The results of this experiment are reported in Figure~\ref{fig:experiment_fullAfullE}. For varying $k$, the observations are very similar to those from the previous experiment, as the cost of the algorithms is dominated by the scaling with respect to $k$, so that the structural differences in $A$ and $E_1,\dots,E_k$ do not heavily influence the results. Starting from $k = 7$, the quadrature based algorithm \texttt{Q} is again the least efficient, and \texttt{AA} turns out to be slightly more efficient than \texttt{HR}. For varying $n$, \texttt{HR} and \texttt{A\added{A}} behave very similarly to what we have observed in the previous experiment, as they do not exploit structure in $A$ or $E_1,\dots,E_k$ in any way. In contrast, our new algorithm \texttt{Q} performs much worse than before. This is in line with the analysis performed in Section~\ref{subsec:efficient_implementation}: As predicted there, the algorithm also scales cubically with respect to $n$ when there is no exploitable structure in the matrices. Still, it is the most efficient of the methods for all considered problem sizes, being about 6 times faster than \texttt{HR} and about 2.5 times faster than \texttt{AA}.
\end{experiment}

\begin{experiment}\label{ex:anymatrix}
\begin{figure}
\centering
\tikzsetnextfilename{anymatrix_accuracy}
\pgfplotsset{height=0.4\linewidth,width=0.875\linewidth,compat=1.10,every axis/.append style={legend style={/tikz/every even column/.append style={column sep=6pt}}}}
\pgfplotsset{every tick label/.append style={font=\small}}
%\tikzexternaldisable

\noindent%
\begin{tikzpicture}[scale=1]%
    \begin{semilogyaxis}[legend style={at={(.99,0.3)}}, 
   	anchor=south east, legend columns=1,cycle list name=list_mod3, 
   	xmin=0.5, xmax=27.5,grid=major, 
   	xlabel={}, ylabel={relative accuracy}, 
   	title={Comparison of accuracy of different algorithms},
   	xticklabels={\texttt{\scriptsize gallery/wathen},
\texttt{\scriptsize gallery/jordbloc},
\texttt{\scriptsize gallery/kms},
\texttt{\scriptsize gallery/parter},
\texttt{\scriptsize contest/pagerank},
\texttt{\scriptsize core/rschur},
\texttt{\scriptsize core/kms\_nonsymm},
\texttt{\scriptsize gallery/condex},
\texttt{\scriptsize gallery/grcar},
\texttt{\scriptsize core/hessmaxdet},
\texttt{\scriptsize gallery/pei},
\texttt{\scriptsize gallery/poisson},
\texttt{\scriptsize core/wilson},
\texttt{\scriptsize gallery/hanowa},
\texttt{\scriptsize gallery/invhess},
\texttt{\scriptsize gallery/gcdmat},
\texttt{\scriptsize gallery/minij},
\texttt{\scriptsize gallery/lesp},
\texttt{\scriptsize core/symmstoch},
\texttt{\scriptsize gallery/lehmer},
\texttt{\scriptsize gallery/randcorr},
\texttt{\scriptsize gallery/tridiag},
\texttt{\scriptsize gallery/dorr},
\texttt{\scriptsize contest/mht},
\texttt{\scriptsize regtools/deriv2},
\texttt{\scriptsize gallery/kahan},
\texttt{\scriptsize gallery/wilk}},xtick={1,...,29}, x tick label style={rotate=45,anchor=east}]

\addplot+[only marks] table [x ={ind},y ={acc_AA}] {figs/dat/test_anymatrix_invsqrt.dat};\addlegendentry{\small \texttt{AA}}
\addplot+[only marks] table [x ={ind},y ={acc_Q32}] {figs/dat/test_anymatrix_invsqrt.dat};\addlegendentry{\small \texttt{Q}, $m = 32$}
\addplot+[only marks] table [x ={ind},y ={acc_Q96}] {figs/dat/test_anymatrix_invsqrt.dat};\addlegendentry{\small \texttt{Q}, $m = 96$}

\end{semilogyaxis}
\end{tikzpicture}

\tikzsetnextfilename{anymatrix_time}
\pgfplotsset{height=0.4\linewidth,width=0.875\linewidth,compat=1.10,every axis/.append style={legend style={/tikz/every even column/.append style={column sep=6pt}}}}
\pgfplotsset{every tick label/.append style={font=\small}}
%\tikzexternaldisable

\noindent%
\begin{tikzpicture}[scale=1]%
    \begin{axis}[legend style={at={(.99,0.795)}}, 
   	anchor=north east, legend columns=1,cycle list name=list_std, 
   	xmin=0.5, xmax=27.5,grid=major, 
   	xlabel={}, ylabel={time [s]}, 
   	title={Comparison of run time of different algorithms},
   	xticklabels={\texttt{\scriptsize gallery/wathen},
\texttt{\scriptsize gallery/jordbloc},
\texttt{\scriptsize gallery/kms},
\texttt{\scriptsize gallery/parter},
\texttt{\scriptsize contest/pagerank},
\texttt{\scriptsize core/rschur},
\texttt{\scriptsize core/kms\_nonsymm},
\texttt{\scriptsize gallery/condex},
\texttt{\scriptsize gallery/grcar},
\texttt{\scriptsize core/hessmaxdet},
\texttt{\scriptsize gallery/pei},
\texttt{\scriptsize gallery/poisson},
\texttt{\scriptsize core/wilson},
\texttt{\scriptsize gallery/hanowa},
\texttt{\scriptsize gallery/invhess},
\texttt{\scriptsize gallery/gcdmat},
\texttt{\scriptsize gallery/minij},
\texttt{\scriptsize gallery/lesp},
\texttt{\scriptsize core/symmstoch},
\texttt{\scriptsize gallery/lehmer},
\texttt{\scriptsize gallery/randcorr},
\texttt{\scriptsize gallery/tridiag},
\texttt{\scriptsize gallery/dorr},
\texttt{\scriptsize contest/mht},
\texttt{\scriptsize regtools/deriv2},
\texttt{\scriptsize gallery/kahan},
\texttt{\scriptsize gallery/wilk}},xtick={1,...,29}, x tick label style={rotate=45,anchor=east}]

\addplot+[only marks] table [x ={ind},y ={time_HR}] {figs/dat/test_anymatrix_invsqrt.dat};\addlegendentry{\small \texttt{HR}}
\addplot+[only marks] table [x ={ind},y ={time_AA}] {figs/dat/test_anymatrix_invsqrt.dat};\addlegendentry{\small \texttt{AA}}
\addplot+[only marks] table [x ={ind},y ={time_Q32}] {figs/dat/test_anymatrix_invsqrt.dat};\addlegendentry{\small \texttt{Q}, $m = 32$}
\addplot+[only marks] table [x ={ind},y ={time_Q96}] {figs/dat/test_anymatrix_invsqrt.dat};\addlegendentry{\small \texttt{Q}, $m = 96$}

\end{axis}
\end{tikzpicture}
\caption{Accuracy and run time comparison for different, matrices from the \texttt{anymatrix} collection, many of which are non-normal. We approximate $L_{z^{-1/2}}(A, E_1, \dots, E_k)$ with $k = 4$ and $n = 300$ (for those matrices which are freely scalable in size). Details on the test matrices are reported in~\ref{sec:appendix_testset}}.\label{fig:experiment_anymatrix}
\end{figure}
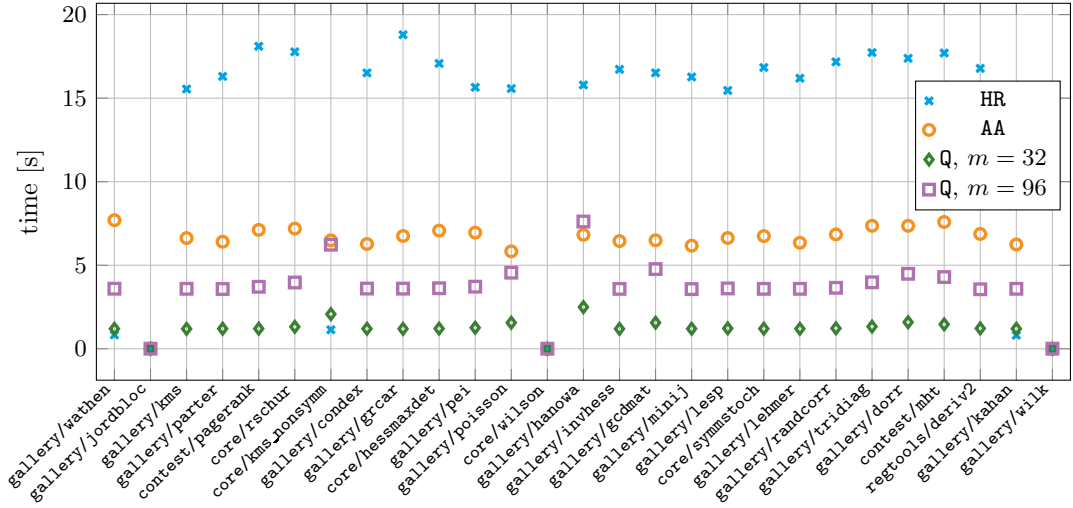
In our last experiment, we test the accuracy of our method on a wide variety of matrices from the \texttt{anymatrix} collection~\cite{Anymatrix}, \replaced{many}{most} of which are non-Hermitian and non-normal. See~\ref{sec:appendix_testset} for details on the data set that is used.

The function we are interested in is $f(z)=z^{-1/2}$ and in the quadrature-based algorithm, we use Gauss-Jacobi quadrature applied to the Stieltjes representation of $z^{-1/2}$ after applying a variable transformation that maps the positive real axis to $[-1,1]$; see~\ref{sec:appendix_quadrature} for details. \replaced{In the first part of the experiment, we}{We always} use $k = 4$ and construct $E_1, \dots, E_k$ as matrices with normally distributed random entries. For matrices with scalable sizes, we choose $n=300$. Details on the exact matrix sizes can be found in Table~\ref{tab:anymatrix}. To also illustrate how the number of quadrature nodes influences performance, we test our quadrature algorithm with two different numbers of quadrature nodes, $m = 32$ and $m = 96$. The results of this experiment are summarized in Figure~\ref{fig:experiment_anymatrix}. The top plot reports the relative accuracy
\begin{equation*}
    \delta := \frac{\|L_{z^{-1/2}}(A, E_1, \dots, E_k) - \widetilde{L}\|_F}{\|L_{z^{-1/2}}(A, E_1, \dots, E_k)\|_F},
\end{equation*}
where $\widetilde{L}$ can either be the approximation from the \texttt{AA} algorithm or from our quadrature-based algorithm. We use the result returned by the algorithm \texttt{HR} as ground truth solution $L_{z^{-1/2}}(A, E_1, \dots, E_k)$. To be able to do this, we removed some highly ill-conditioned matrices from the data set for which it also does not produce an accurate solution, see~\ref{sec:appendix_testset} for details. We can observe that for $m = 32$ quadrature nodes, our algorithm reaches---on average---a comparable accuracy to \texttt{AA}. For some matrices, it produces much better results, while it is also sometimes worse by several orders of magnitude. With $m = 96$ quadrature nodes, the results are much more accurate than those of \texttt{AA} for most matrices, except for a few rather ill-conditioned matrices (for which \texttt{AA} also reaches lower accuracy than for the other examples). The run times of the different algorithms (now also including \texttt{HR}) are depicted in the bottom plot of Figure~\ref{fig:experiment_anymatrix}. We can observe results that are very similar to those of previous experiments. Algorithm \texttt{HR} requires the most time, while our quadrature algorithm outperforms \texttt{AA} by factors of about 2-3 (for $m=96$) or about $5-6$ (for $m = 32$). Note that three test matrices have a very small size (and are not scalable), explaining the very low run times reported for them.

\begin{figure}
\centering
\tikzsetnextfilename{anymatrix_accuracy2}
\pgfplotsset{height=0.4\linewidth,width=0.875\linewidth,compat=1.10,every axis/.append style={legend style={/tikz/every even column/.append style={column sep=6pt}}}}
\pgfplotsset{every tick label/.append style={font=\small}}
%\tikzexternaldisable

\noindent%
\begin{tikzpicture}[scale=1]%
    \begin{semilogyaxis}[legend style={at={(.99,0.3)}}, 
   	anchor=south east, legend columns=1,cycle list name=list_mod3, 
   	xmin=0.5, xmax=13.5,grid=major, 
   	xlabel={}, ylabel={relative accuracy}, 
   	title={Comparison of accuracy of different algorithms},
   	xticklabels={\texttt{\scriptsize core/rschur},
    \texttt{\scriptsize gallery/hanowa},
    \texttt{\scriptsize gallery/jordbloc},
    \texttt{\scriptsize gallery/lesp},
    \texttt{\scriptsize core/kms\_nonsymm},
    \texttt{\scriptsize gallery/grcar},
    \texttt{\scriptsize contest/pagerank},
    \texttt{\scriptsize gallery/parter},
    \texttt{\scriptsize core/hessmaxdet},
    \texttt{\scriptsize gallery/invhess},
    \texttt{\scriptsize contest/mht},
    \texttt{\scriptsize gallery/dorr},
    \texttt{\scriptsize gallery/kahan}},xtick={1,...,29}, x tick label style={rotate=45,anchor=east}]

\addplot+[only marks] table [x ={ind},y ={acc_AA}] {figs/dat/test_anymatrix_invsqrt2.dat};\addlegendentry{\small \texttt{AA}}
\addplot+[only marks] table [x ={ind},y ={acc_Q32}] {figs/dat/test_anymatrix_invsqrt2.dat};\addlegendentry{\small \texttt{Q}, $m = 32$}
\addplot+[only marks] table [x ={ind},y ={acc_Q96}] {figs/dat/test_anymatrix_invsqrt2.dat};\addlegendentry{\small \texttt{Q}, $m = 96$}

\end{semilogyaxis}
\end{tikzpicture}

\tikzsetnextfilename{anymatrix_time2}
\pgfplotsset{height=0.4\linewidth,width=0.875\linewidth,compat=1.10,every axis/.append style={legend style={/tikz/every even column/.append style={column sep=6pt}}}}
\pgfplotsset{every tick label/.append style={font=\small}}
%\tikzexternaldisable

\noindent%
\begin{tikzpicture}[scale=1]%
    \begin{axis}[legend style={at={(.99,0.975)}}, 
   	anchor=north east, legend columns=1,cycle list name=list_std, 
   	xmin=0.5, xmax=13.5,grid=major, 
   	xlabel={}, ylabel={time [s]}, 
   	title={Comparison of run time of different algorithms},
   	xticklabels={\texttt{\scriptsize core/rschur},
    \texttt{\scriptsize gallery/hanowa},
    \texttt{\scriptsize gallery/jordbloc},
    \texttt{\scriptsize gallery/lesp},
    \texttt{\scriptsize core/kms\_nonsymm},
    \texttt{\scriptsize gallery/grcar},
    \texttt{\scriptsize contest/pagerank},
    \texttt{\scriptsize gallery/parter},
    \texttt{\scriptsize core/hessmaxdet},
    \texttt{\scriptsize gallery/invhess},
    \texttt{\scriptsize contest/mht},
    \texttt{\scriptsize gallery/dorr},
    \texttt{\scriptsize gallery/kahan}},xtick={1,...,29}, x tick label style={rotate=45,anchor=east}]

\addplot+[only marks] table [x ={ind},y ={time_HR}] {figs/dat/test_anymatrix_invsqrt2.dat};\addlegendentry{\small \texttt{HR}}
\addplot+[only marks] table [x ={ind},y ={time_AA}] {figs/dat/test_anymatrix_invsqrt2.dat};\addlegendentry{\small \texttt{AA}}
\addplot+[only marks] table [x ={ind},y ={time_Q32}] {figs/dat/test_anymatrix_invsqrt2.dat};\addlegendentry{\small \texttt{Q}, $m = 32$}
\addplot+[only marks] table [x ={ind},y ={time_Q96}] {figs/dat/test_anymatrix_invsqrt2.dat};\addlegendentry{\small \texttt{Q}, $m = 96$}

\end{axis}
\end{tikzpicture}
\caption{\added{Accuracy and run time comparison for different, non-symmetric matrices from the \texttt{anymatrix} collection. We approximate $L_{z^{-1/2}}(A, E_1, E_2)$, where $E_1, E_2$ are outer products of randomly chosen canonical unit vectors and $n = 300$. Details on the test matrices are reported in~\ref{sec:appendix_testset}}}.\label{fig:experiment_anymatrix2}
\end{figure}
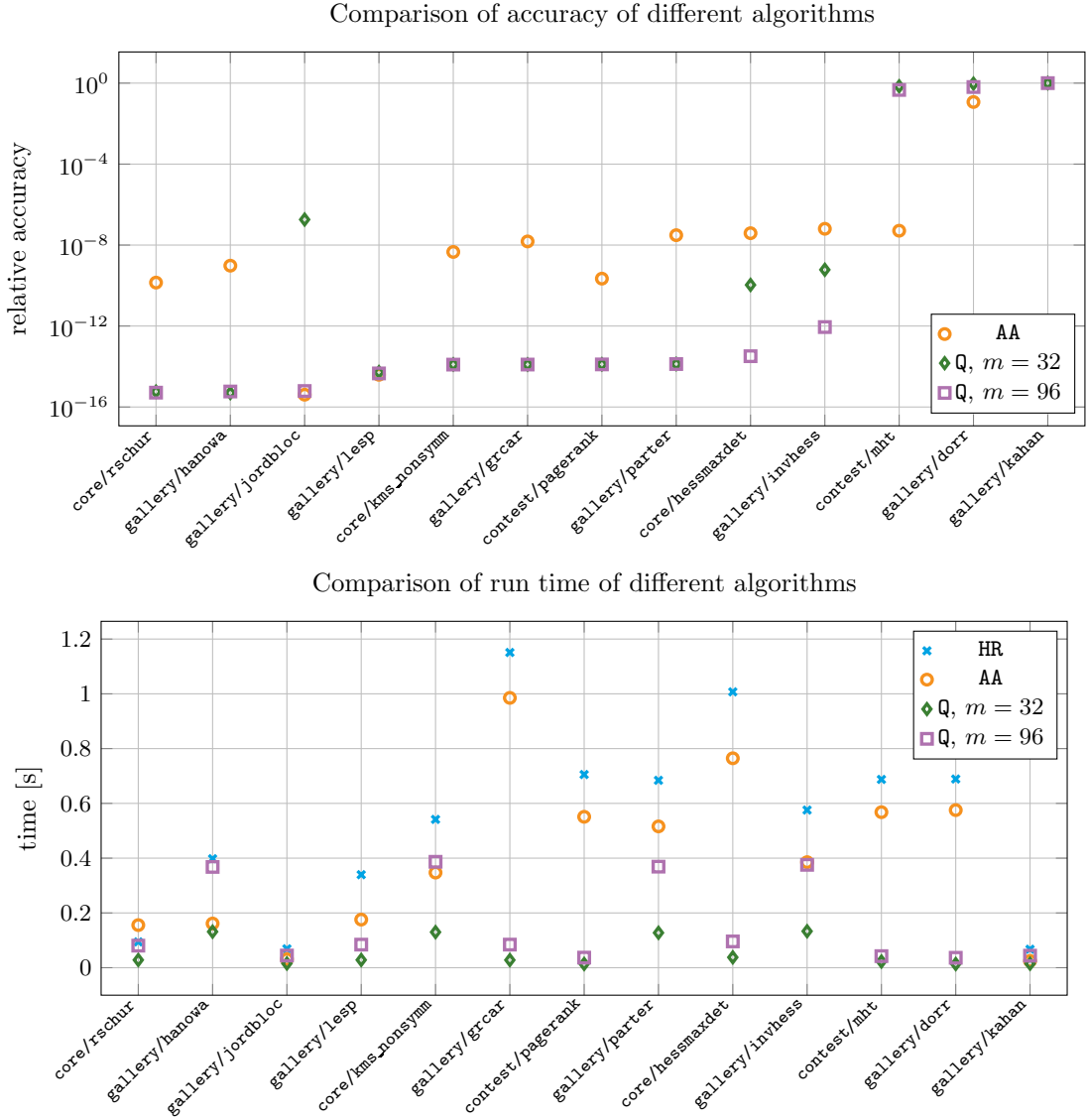

\added{As the most important practical application of higher-order Fr\'echet derivatives is bounding the level-2 condition number, we now turn to the specifics of this application in the second part of this experiment. Bounding the level-2 condition number via~\eqref{eq:lvl2_kronecker} requires computing the Kronecker form of the second-order Fr\'echet derivative, which is possible by~\cite[Algorithm~4.2]{HighamRelton2014}. This algorithm requires repeatedly evaluating $L_f^{(2)}(A, E_1, E_2)$ with $E_1$ and $E_2$ chosen as outer products of all possible combinations of canonical unit vectors. Thus, for simulating one substep of this computation, for the following experiment, we set $k = 2$ and choose $E_i = \ve_{\alpha_i}\ve_{\beta_i}^T$ for random $\alpha_i,\beta_i \in \{1,\dots,n\}$. As Theorem~\ref{the:lvl2_stieltjes_bound} gives an explicit formula for the level-2 condition number of the inverse square root of Hermitian positive definite matrices (and it is therefore not necessary to use the bound~\eqref{eq:lvl2_kronecker} based on the Kronecker form), we only use the non-Hermitian matrices from our test set in this second part of the experiment.}

\added{Apart from the adaptions mentioned above, we use the same experimental setup as before and report the same performance indicators. Results are depicted in Figure~\ref{fig:experiment_anymatrix2}. The results are quite similar to before in the sense that we again observe that \texttt{HR} has the highest run time and that the quadrature algorithm is (almost) always faster than \texttt{AA}, while mostly reaching a comparable (or better) accuracy. What has significantly changed in comparison to the first part of the experiment is the margin by which our new algorithm outperforms \texttt{AA}. While the run time of the different algorithms was very similar for most matrices in the test set before, there is a lot more variety now, especially for the quadrature algorithms. This is due to the fact that the run time of the method largely depends on how efficiently linear systems with $A$ can be solved; see also the discussion in Section~\ref{subsec:efficient_implementation}. For the test matrices considered in this experiment, algorithm~\texttt{Q} with $m=32$ needs between a factor of $2$ and $40$ less computation time than \texttt{AA}. Let us also note that for condition number estimation, one does typically not require a high relative accuracy (in particular as the computed quantity is only a bound for the actual level-2 condition number anyway), so that accuracy to one or two significant digits is often enough and one could also use significantly fewer quadrature nodes, further speeding up the computation.}
\end{experiment}

\section{Conclusions}\label{sec:concluding_remarks}
We have derived new integral representations for the higher-order \Fd{}s of several classes of matrix functions. This integral representation allowed to find an exact formula for the level-2 condition number of $f(A)$ for Hermitian $A$ and a large class of functions $f$ with strictly monotonic derivative. Further, we used the integral representation to devise a computational method for approximating the higher-order \Fd\ based on numerical quadrature. In numerical experiments, we have demonstrated that for moderate derivative orders $k$, our algorithm improves upon the algorithms available in the literature so far. In particular for the important special case of rank-one direction terms---which, e.g., occurs when computing the Kronecker form of the higher-order \Fd{}---we obtain significant savings. 

There are several directions for future research. On the one hand, the generalization of Krylov subspace methods for the first-order \Fd\ with low-rank direction term~\cite{KandolfKoskelaReltonSchweitzer2021, Kressner2019} to the higher-order case would be interesting to further reduce computational complexity in this setting. On the other hand, it is worth investigating whether the proposed integral representations of the second-order \Fd\ also allow to find formulas or at least bounds for level-2 condition numbers of other classes of functions and/or matrices then what we have considered in the present work.

\section*{Acknowledgment}
The author wishes to thank Bahar Arslan for fruitful discussions on the topic and for her careful reading of an earlier version of this manuscript \added{as well as two anonymous referees for their helpful suggestions which improved the manuscript}.

\bibliography{matrixfunctions}

\newcommand{\noopsort}[1]{} \newcommand{\printfirst}[2]{#1}
  \newcommand{\singleletter}[1]{#1} \newcommand{\switchargs}[2]{#2#1}
\begin{thebibliography}{10}
\expandafter\ifx\csname url\endcsname\relax
  \def\url#1{\texttt{#1}}\fi
\expandafter\ifx\csname urlprefix\endcsname\relax\def\urlprefix{URL }\fi
\expandafter\ifx\csname href\endcsname\relax
  \def\href#1#2{#2} \def\path#1{#1}\fi

\bibitem{EstradaHigham2010}
E.~Estrada, D.~J. Higham, Network properties revealed through matrix functions,
  SIAM Rev. 52~(4) (2010) 696--714.

\bibitem{HochbruckLubich1997}
M.~Hochbruck, {\relax Ch}.~Lubich, On {K}rylov subspace approximations to the
  matrix exponential operator, SIAM J.\ Numer.\ Anal. 34~(5) (1997) 1911--1925.

\bibitem{HochbruckOstermann2010}
M.~Hochbruck, A.~Ostermann, Exponential integrators, Acta Numerica 19 (2010)
  209--286.

\bibitem{HochbruckLubichSelhofer1998}
M.~Hochbruck, {\relax Ch}.~Lubich, H.~Selhofer, Exponential integrators for
  large systems of differential equations, SIAM J.\ Sci.\ Comput. 19~(5) (1998)
  1552--1574.

\bibitem{VanDenEshofFrommerLippertSchillingVanDerVorst2002}
J.~{\noopsort{Eshof}van den Eshof}, A.~Frommer, {\relax Th}.~Lippert,
  K.~Schilling, H.~A. van~der Vorst, Numerical methods for the {QCD} overlap
  operator. {I}. {S}ign-function and error bounds, Comput.\ Phys.\ Commun.
  146~(2) (2002) 203--224.

\bibitem{Neuberger1998}
H.~Neuberger, Exactly massless quarks on the lattice, Phys.\ Lett., B
  417~(1--2) (1998) 141--144.

\bibitem{IlicTurnerSimpson2010}
M.~Ili\'c, I.~W. Turner, D.~P. Simpson, A restarted {L}anczos approximation to
  functions of a symmetric matrix, IMA J.\ Numer.\ Anal. 30~(4) (2010)
  1044--1061.

\bibitem{HighamRelton2014}
N.~J. Higham, S.~D. Relton, Higher order {F}r{\'e}chet derivatives of matrix
  functions and the level-2 condition number, SIAM J.\ Matrix Anal.\ Appl.
  35~(3) (2014) 1019--1037.

\bibitem{Higham2008}
N.~J. Higham, Functions of Matrices: Theory and Computation, SIAM,
  Philadelphia, 2008.

\bibitem{AmatBusquierGutierrez2003}
S.~Amat, S.~Busquier, J.~Guti{\'e}rrez, Geometric constructions of iterative
  functions to solve nonlinear equations, J.\ Comput.\ Appl.\ Math. 157~(1)
  (2003) 197--205.

\bibitem{KandolfRelton2017}
P.~Kandolf, S.~D. Relton, A block {K}rylov method to compute the action of the
  {F}r{\'e}chet derivative of a matrix function on a vector with applications
  to condition number estimation, SIAM J.\ Sci.\ Comput. 39~(4) (2017)
  A1416--A1434.

\bibitem{KandolfKoskelaReltonSchweitzer2021}
P.~Kandolf, A.~Koskela, S.~D. Relton, M.~Schweitzer, Computing low-rank
  approximations of the {F}r{\'e}chet derivative of a matrix function using
  {K}rylov subspace methods, Numer.\ Linear Algebra Appl. 28~(6) (2021) e2401.

\bibitem{AlzerBerg2002}
H.~Alzer, C.~Berg, Some classes of completely monotonic functions, Ann.\ Acad.\
  Sci.\ Fenn., Math. 27 (2002) 445--460.

\bibitem{Berg2007}
C.~Berg, {S}tieltjes-{P}ick-{B}ernstein-{S}choenberg and their connection to
  complete monotonicity, in: J.~Mateu, E.~Porcu (Eds.), Positive Definite
  Functions. From Schoenberg to Space-Time Challenges, Dept.\ of Mathematics,
  University Jaume I, Castell\'{o}n de la Plana, Spain, 2008, pp. 15--45.

\bibitem{Henrici1977}
P.~Henrici, Applied and Computational Complex Analysis, Vol. 2, John Wiley \&
  Sons, New York, 1977.

\bibitem{KaluginJeffreyCorlessBorwein2012}
G.~A. Kalugin, D.~J. Jeffrey, R.~M. Corless, P.~B. Borwein, Stieltjes and other
  integral representations for functions of {L}ambert {W}, Integral Transforms
  Spec.\ Funct. 23~(8) (2012) 581--593.

\bibitem{Cardoso2012}
J.~R. Cardoso, Computation of the matrix $p$th root and its {F}r\'echet
  derivative by integrals, Electron.\ Trans.\ Numer.\ Anal. 39 (2012) 414--436.

\bibitem{HornJohnson1991}
R.~A. Horn, {\relax Ch}.~R. Johnson, Topics in Matrix Analysis, Cambridge
  University Press, Cambridge, 1991.

\bibitem{Anymatrix}
N.~J. Higham, M.~Mikaitis, Anymatrix: an extensible {MATLAB} matrix collection,
  Numer.\ Algorithms 90~(3) (2022) 1175--1196.

\bibitem{AlMohyArslan2020}
A.~Al-Mohy, B.~Arslan, The complex step approximation to the higher order
  {F}r{\'e}chet derivatives of a matrix function, Numer.\ Algorithms (2020)
  1--14.

\bibitem{Werner2021}
T.~Werner, On using the complex step method for the approximation of
  {F}r{\'e}chet derivatives of matrix functions in automorphism groups, Tech.
  rep. (2021).
\newblock \href {http://arxiv.org/abs/2112.06786} {\path{arXiv:2112.06786}}.

\bibitem{AlMohyHigham2011}
A.~H. Al-Mohy, N.~J. Higham, Computing the action of the matrix exponential,
  with an application to exponential integrators, SIAM J.\ Sci.\ Comput. 33~(2)
  (2011) 488--511.

\bibitem{FrommerLundSzyld2017}
A.~Frommer, K.~Lund, D.~B. Szyld, Block {K}rylov subspace methods for functions
  of matrices, Electron.\ Trans.\ Numer.\ Anal. 47 (2017) 100--126.

\bibitem{Hale2008}
N.~Hale, N.~J. Higham, L.~N. Trefethen, Computing {$A^\alpha, \log(A)$}, and
  related matrix functions by contour integrals, SIAM J.\ Numer.\ Anal. 46~(5)
  (2008) 2505--2523.

\bibitem{TrefethenWeidemanSchmelzer2006}
L.~N. Trefethen, J.~A.~C. Weideman, T.~Schmelzer, Talbot quadratures and
  rational approximations, BIT 46~(3) (2006) 653--670.

\bibitem{WeidemanTrefethen2007}
J.~A.~C. Weideman, L.~N. Trefethen, Parabolic and hyperbolic contours for
  computing the {B}romwich integral, Math.\ Comput. 76~(259) (2007) 1341--1356.

\bibitem{FrommerGuettelSchweitzer2014a}
A.~Frommer, S.~G{\"u}ttel, M.~Schweitzer, Efficient and stable {A}rnoldi
  restarts for matrix functions based on quadrature, SIAM J.\ Matrix Anal.\
  Appl. 35 (2014) 661--683.

\bibitem{Schweitzer2016}
M.~Schweitzer, Restarting and error estimation in polynomial and extended
  {K}rylov subspace methods for the approximation of matrix functions, Ph.D.
  thesis, Bergische Universit{\"a}t Wuppertal, Fakult{\"a}t f{\"u}r Mathematik
  und Naturwissenschaften (2016).

\bibitem{Kressner2019}
D.~Kressner, A {K}rylov subspace method for the approximation of bivariate
  matrix functions, in: Structured matrices in numerical linear algebra,
  Springer, 2019, pp. 197--214.

\bibitem{Weideman2006}
J.~A.~C. Weideman, Optimizing {T}albot's contours for the inversion of the
  {L}aplace transform, SIAM J.\ Numer.\ Anal. 44~(6) (2006) 2342--2362.

\end{thebibliography}

\appendix

\section{Quadrature rules} \label{sec:appendix_quadrature}
In this section, we brief\/ly discuss important matrix function quadrature rules and how they are applicable in our setting. We also report the results of small numerical experiments that confirm that one can expect about the same accuracy for approximating $L_f(A,E_1,\dots,E_k)$ as for approximating $f(A)\vb$.

\paragraph{Exponential function on the negative real axis}
When approximating the exponential of a matrix $A$ with eigenvalues on (or close to) the nonpositive real axis $(-\infty, 0]$, a popular choice is to use the trapezoidal (or midpoint) rule on a parabolic, hyperbolic or cotangent contour~\cite{Weideman2006,WeidemanTrefethen2007,TrefethenWeidemanSchmelzer2006}. 
\begin{figure}
\centering
\tikzsetnextfilename{experiment_accuracy}
\pgfplotsset{height=0.4\linewidth,width=0.875\linewidth,compat=1.10,every axis/.append style={legend style={/tikz/every even column/.append style={column sep=6pt}}}}
\pgfplotsset{every tick label/.append style={font=\small}}
%\tikzexternaldisable

\noindent%
\begin{tikzpicture}[scale=1]%
    \begin{semilogyaxis}[legend style={at={(0.36,0.33)}}, 
   	anchor=north east, legend columns=1,cycle list name=list_std3, 
   	xmin=2, xmax=38,grid=major, 
   	xlabel={\small $m$}, ylabel={\small accuracy}, 
   	title={Accuracy of quadrature rules from~\cite{TrefethenWeidemanSchmelzer2006} for $L_{\exp}(A,\cdot,\dots,\cdot)$}]

\addplot+[] table [x ={Q},y ={M1}] {figs/dat/experiment_accuracy_exp.dat};\addlegendentry{\small  Parabolic contour~\eqref{eq:parabolic_contour}}
\addplot+[] table [x ={Q},y ={M2}] {figs/dat/experiment_accuracy_exp.dat};\addlegendentry{\small  Hyperbolic contour~\eqref{eq:hyperbolic_contour}}
\addplot+[] table [x ={Q},y ={M3}] {figs/dat/experiment_accuracy_exp.dat};\addlegendentry{\small  Cotangent contour~\eqref{eq:cotangent_contour}}

\addplot+[loosely dashed] table [x ={Q},y ={f1}] {figs/dat/experiment_accuracy_exp.dat};
\addplot+[loosely dashed] table [x ={Q},y ={f2}] {figs/dat/experiment_accuracy_exp.dat};
\addplot+[loosely dashed] table [x ={Q},y ={f3}] {figs/dat/experiment_accuracy_exp.dat};

\end{semilogyaxis}
\end{tikzpicture}
\cprotect\caption{Relative error of approximation~\eqref{eq:frechet_quadrature} for $L_{\exp}(A,\cdot,\dots,\cdot)$ for varying number $m$ of quadrature nodes and the three quadrature rules~\eqref{eq:parabolic_contour},~\eqref{eq:hyperbolic_contour}, and~\eqref{eq:cotangent_contour}. The derivative order is chosen as $k=4$, the matrix $A \in \R^{25 \times 25}$ is generated by the Matlab command \verb|A = gallery('lesp',25)| and the matrices $E_1, \dots, E_4$ have normally distributed random entries. For comparison purposes, the dashed lines show the relative error norms for approximating $\exp(A)\vb$ by the same quadrature rules, where $\vb \in \R^{25}$ is a vector with normally distributed random entries.}\label{fig:experiment_accuracy}
\end{figure}

Using an optimized parabolic contour leads to the nodes and weights
\begin{eqnarray}
\zeta_j &=& m\cdot(0.1309-0.1194\theta_j^2+0.25i\theta_j), \nonumber\\
w_j &=& -\exp(\zeta_j)\cdot(0.2388i\theta_j+0.25),\label{eq:parabolic_contour}
\end{eqnarray}
where $\theta_j, j = 1,\dots,m$ are equispaced points in $[-\pi,\pi]$. This quadrature rule achieves a convergence rate of $\mathcal{O}(2.85^{-m})$. Similarly, a quadrature rule on an optimized hyperbolic contour that achieves a convergence rate of $\mathcal{O}(3.2^{-m})$ is given by
\begin{eqnarray}
\zeta_j &=& 2.246m\cdot(1-\sin(1.1721-0.3443i\theta_j), \nonumber\\
w_j &=& -0.7733\exp(\zeta_j)\cos(1.1721-0.3443i\theta_j),\label{eq:hyperbolic_contour}
\end{eqnarray}
and a rule on a cotangent contour that yields a convergence rate of $\mathcal{O}(3.89^{-m})$ is given by
\begin{eqnarray}
\zeta_j &=& m\cdot(0.5017\theta_j\cot(0.6407\theta_j)-0.6122+0.2645i\theta_j), \nonumber\\
w_j &=& \exp(\zeta_j)\left(0.5017i\cot(0.6407\theta_j)-0.3214i\theta_j\csc(0.6407\theta_j)^2-0.2645\right).\label{eq:cotangent_contour}
\end{eqnarray}
When $A$ is real, by symmetry it is sufficient to evaluate only half of the $m$ quadrature points for all of the above quadrature rules~\cite{TrefethenWeidemanSchmelzer2006}. 

We investigate the accuracy of the approximations resulting from these three choices of quadrature rules in a small numerical experiment, where we use the same matrices as in Experiment~\ref{ex:full}: We generate $A \in \R^{25 \times 25}$ by the MATLAB command \verb|A = gallery('lesp',25)|, which results in a tridiagonal matrix with smoothly distributed eigenvalues in the interval $[-53.5, -4.5]$, and the matrices $E_1,\dots,E_4$ are chosen to be dense matrices with normally distributed random entries. Figure~\ref{fig:experiment_accuracy} depicts the relative error norm for the approximation~\eqref{eq:frechet_quadrature} when using the nodes and weights~\eqref{eq:parabolic_contour},~\eqref{eq:hyperbolic_contour} or~\eqref{eq:cotangent_contour}. The reference solution is computed using~\cite[Algorithm 3.6]{HighamRelton2014} employing the built-in MATLAB function \texttt{expm}. % executed in quadruple precision. 
We observe that both the hyperbolic contour~\eqref{eq:hyperbolic_contour} and the cotangent contour~\eqref{eq:cotangent_contour} show signs of instability starting at around $m = 16$ quadrature nodes and do not manage to reach high accuracy. Using the parabolic contour, an approximation error close to the order of machine precision can be reached. To investigate whether the observed instabilities are an inherent drawback of the proposed integral representation of $L_{\exp}(A,\cdot,\dots,\cdot)$, we also report the relative errors reached by the same quadrature rules when approximating $\exp(A)\vb$, where $\vb \in \R^{25}$ is a vector with normally distributed random entries. These are depicted by dashed lines. One can clearly observe that the behavior is strikingly similar, i.e., the instabilities seem to be inherent to the quadrature rule and not related to the specific nature of the integral~\eqref{eq:integral_representation} we want to approximate. Similar observations concerning the accuracy of the rules~\eqref{eq:parabolic_contour}--\eqref{eq:cotangent_contour} for increasing number of quadrature nodes have been reported in~\cite[Section~4.3]{FrommerGuettelSchweitzer2014a}. In summary, this example illustrates that we can expect a similar performance as in quadrature rules for $\exp(A)\vb$, which also means that sometimes the obtainable accuracy might be limited if a suboptimal choice of quadrature rule is made.

\paragraph{Square root of matrices with positive eigenvalues}
\begin{figure}
\centering
\tikzsetnextfilename{experiment_accuracy_sqrt}
\pgfplotsset{height=0.4\linewidth,width=0.875\linewidth,compat=1.10,every axis/.append style={legend style={/tikz/every even column/.append style={column sep=6pt}}}}
\pgfplotsset{every tick label/.append style={font=\small}}
%\tikzexternaldisable

\noindent%
\begin{tikzpicture}[scale=1]%
    \begin{semilogyaxis}[legend style={at={(0.995,0.98)}}, 
   	anchor=north east, legend columns=1,cycle list name=list_std3, 
   	xmin=2, xmax=29,grid=major, 
   	xlabel={\small $m$}, ylabel={\small accuracy}, 
   	title={Accuracy of quadrature rules from~\cite{Hale2008} for $L_{z^{1/2}}(A,\cdot,\dots,\cdot)$}]

\addplot+[] table [x ={Q},y ={M1}] {figs/dat/experiment_accuracy_sqrt.dat};\addlegendentry{\small  Method 1}
\addplot+[] table [x ={Q},y ={M2}] {figs/dat/experiment_accuracy_sqrt.dat};\addlegendentry{\small  Method 2}
\addplot+[] table [x ={Q},y ={M3}] {figs/dat/experiment_accuracy_sqrt.dat};\addlegendentry{\small  Method 3}

\addplot+[loosely dashed] table [x ={Q},y ={f1}] {figs/dat/experiment_accuracy_sqrt.dat};
\addplot+[loosely dashed] table [x ={Q},y ={f2}] {figs/dat/experiment_accuracy_sqrt.dat};
\addplot+[loosely dashed] table [x ={Q},y ={f3}] {figs/dat/experiment_accuracy_sqrt.dat};

\end{semilogyaxis}
\end{tikzpicture}
\cprotect\caption{Relative error of approximation~\eqref{eq:frechet_quadrature} for $L_{z^{1/2}}(A,\cdot,\dots,\cdot)$ for varying number $m$ of quadrature nodes and the three quadrature rules from~\cite{Hale2008}. The derivative order is chosen as $k=4$, the matrix $A \in \R^{25 \times 25}$ is generated by the Matlab command \verb|A = -gallery('lesp',25)| and the matrices $E_1, \dots, E_4$ have normally distributed random entries. For comparison purposes, the dashed lines show the relative error norms for approximating $A^{1/2}\vb$ by the same quadrature rules, where $\vb \in \R^{25}$ is a vector with normally distributed random entries.}\label{fig:experiment_accuracy_sqrt}
\end{figure}
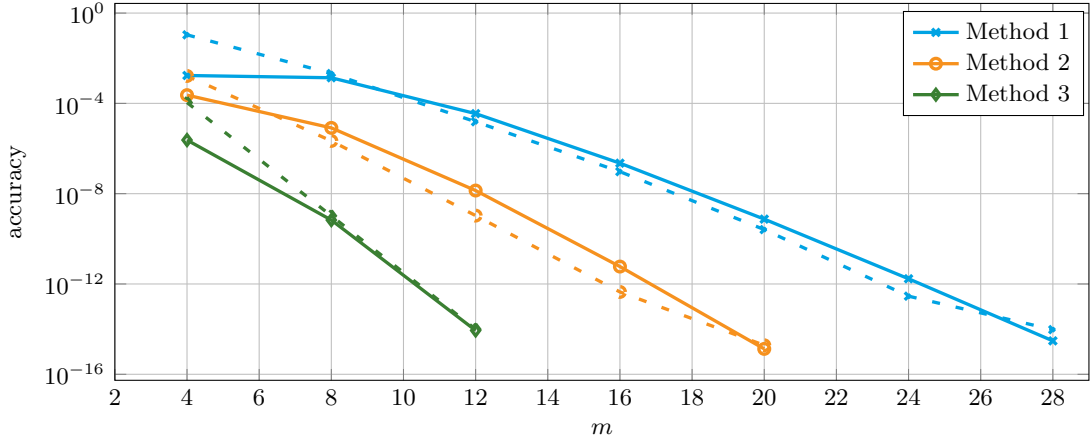
Another matrix function of very high practical importance is the matrix square root $f(A) = A^{1/2}$. In~\cite{Hale2008}, different quadrature approximations for $A^{1/2}$ (among other functions) are discussed, based on conformal mappings involving Jacobi elliptic functions combined with the trapezoidal rule (or the midpoint rule). 
These rules (referred to simply as ``Method 1'', ``Method 2'' and ``Method 3'' in~\cite{Hale2008}) converge with geometric rates of $\mathcal{O}(e^{-\pi^2 m / (\log (\lmax/\lmin)+3)}$), $\mathcal{O}(e^{-2\pi^2 m / (\log (\lmax/\lmin)+6)})$ and $\mathcal{O}(e^{-2\pi^2 m / (\log (\lmax/\lmin)+3)})$, respectively. We do not discuss the rules or their derivations in detail here, but just briefly explain why they are directly applicable in our setting and then illustrate by a small experiment that one can again expect performance to be similar to that of approximating $A^{1/2}$ or $A^{1/2}\vb$. The first step in the derivation of the quadrature rules in~\cite{Hale2008} consists of rewriting the Cauchy integral representation of the square root as
\begin{equation}\label{eq:cauchy_integral_sqrt}
    A^{1/2} = \frac{A}{2\pi i} \int_\Gamma \zeta^{-1/2} (\zeta I-A) \d \zeta.
\end{equation}
By considerations similar to those in Theorem~\ref{the:stieltjes_times_z}, this leads to the representation
\begin{equation}\label{eq:integral_representation_sqrt_hale}
L_f^{(k)}(A, E_1, \dots E_k) = \frac{1}{2\pi i} \int_\Gamma \sum\limits_{\pi \in S_k} \zeta \cdot\zeta^{-1/2} M_\pi(\zeta; A, E_1, \dots E_k) \d \zeta = \frac{1}{2\pi i} \int_\Gamma \sum\limits_{\pi \in S_k} \zeta^{1/2} M_\pi(\zeta; A, E_1, \dots E_k) \d \zeta.
\end{equation}
Thus, the representation of the higher-order \Fd\ is not changed at all by this transformation, so that no special techniques or transformations are necessary to apply the corresponding quadrature rules. We now test the three methods using the same experimental setup as for the exponential in the preceding paragraph, just with $f(z) = z^{1/2}$ instead of $f(z) = \exp(z)$ and replacing $A$ by $-A$ to obtain a positive definite matrix. The results are reported in Figure~\ref{fig:experiment_accuracy_sqrt}. The dashed lines again correspond to the accuracy of the same quadrature rules when applied for approximating $A^{1/2}\vb$, where $\vb$ is a vector with normally distributed random entries. We again observe that the behavior when approximating the higher-order \Fd{} is very similar to that when approximating the action of a matrix function on a vector, as we already noted for the exponential. All three quadrature rules allow to find approximations near the order of machine precision.

\paragraph{Inverse square root}
\begin{figure}
\centering
\tikzsetnextfilename{experiment_accuracy_invsqrt}
\pgfplotsset{height=0.4\linewidth,width=0.875\linewidth,compat=1.10,every axis/.append style={legend style={/tikz/every even column/.append style={column sep=6pt}}}}
\pgfplotsset{every tick label/.append style={font=\small}}
%\tikzexternaldisable

\noindent%
\begin{tikzpicture}[scale=1]%
    \begin{semilogyaxis}[legend style={at={(0.995,0.98)}}, 
   	anchor=north east, legend columns=1,cycle list name=list_std1, 
   	xmin=2, xmax=41,grid=major, 
   	xlabel={\small $m$}, ylabel={\small accuracy}, 
   	title={Accuracy of quadrature rule from~\cite{FrommerGuettelSchweitzer2014a} for $L_{z^{-1/2}}(A,\cdot,\dots,\cdot)$}]

\addplot+[] table [x ={Q},y ={M1}] {figs/dat/experiment_accuracy_invsqrt.dat};\addlegendentry{\small  Gauss--Jacobi quadrature}

\addplot+[loosely dashed] table [x ={Q},y ={f1}] {figs/dat/experiment_accuracy_invsqrt.dat};

\end{semilogyaxis}
\end{tikzpicture}
\cprotect\caption{Relative error of approximation~\eqref{eq:frechet_quadrature} for $L_{z^{-1/2}}(A,\cdot,\dots,\cdot)$ for varying number $m$ of quadrature nodes and the quadrature rule~\eqref{eq:invsqrt_stieltjes_quadrature}. The derivative order is chosen as $k=4$, the matrix $A \in \R^{25 \times 25}$ is generated by the Matlab command \verb|A = -gallery('lesp',25)| and the matrices $E_1, \dots, E_4$ have normally distributed random entries. For comparison purposes, the dashed line shows the relative error norms for approximating $A^{-1/2}\vb$ by the same quadrature rule, where $\vb \in \R^{25}$ is a vector with normally distributed random entries.}\label{fig:experiment_accuracy_invsqrt}
\end{figure}
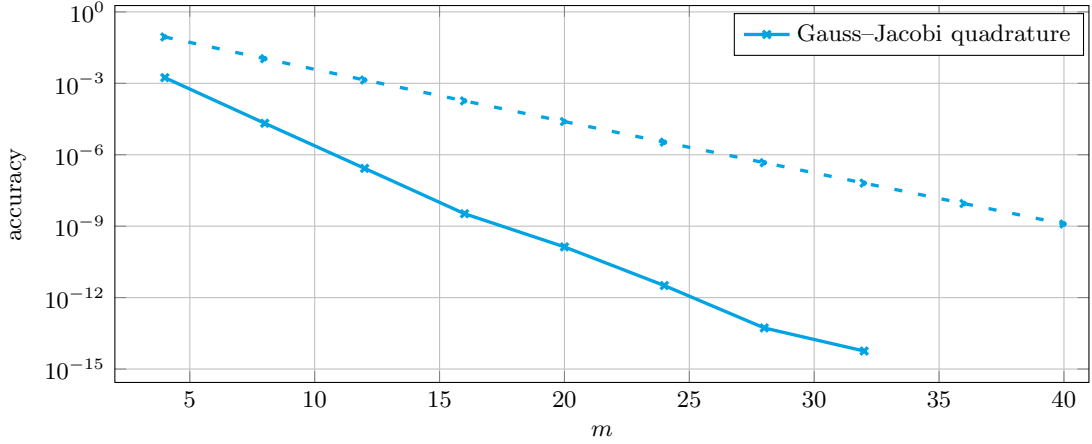
The last example we consider in this section is the inverse square root $f(z) = z^{-1/2}$ . A particular convenient way of approximating it by numerical quadrature arises from its Stieltjes representation
\begin{equation}\label{eq:inverse_square_root_stieltjes}
    z^{-1/2} = \frac{1}{\pi}\int_0^\infty \frac{1}{\sqrt{t}} \cdot \frac{1}{z+t} \d t.
\end{equation}
The big advantage of this representation compared to other integral representations of $z^{-1/2}$ is that the path of integration is always the positive real axis and thus, in the context of approximating $f(A)$ or its \Fd{}, in particular independent of the spectrum of $A$. This way, it is also applicable if spectral information on $A$ is not available or very costly to compute. In~\cite{FrommerGuettelSchweitzer2014a}, it is proposed to apply the Cayley transform $t = \frac{1+x}{1-x}$ to the integral~\eqref{eq:inverse_square_root_stieltjes}, which yields
\begin{equation}\label{eq:inverse_square_root_stieltjes_transformed}
    f(z) = \frac{2}{\pi}\int_{-1}^1 (1-x)^{-1/2}(1+x)^{-1/2} \cdot \frac{1}{(1+x)+z(1-x)} \d t.
\end{equation}
The integral~\eqref{eq:inverse_square_root_stieltjes_transformed} over the finite interval $[-1, 1]$ can be approximated efficiently using Gauss--Jacobi quadrature which allows to resolve the weight function $(1-x)^{-1/2}(1+x)^{-1/2}$ exactly. In our setting, utilizing this approach leads to a representation
\begin{eqnarray}
L_f^{(k)}(A, E_1, \dots E_k) &=& \frac{2}{\pi} \int_{-1}^1 \sum\limits_{\pi \in S_k} (1-x)^{-1/2}(1+x)^{-1/2} \frac{1}{1-x} \cdot M_\pi\left(\frac{1-x}{1+x}; A, E_1,\dots,E_k\right) \d t \nonumber\\
&\approx& \frac{2}{\pi} \sum_{i = 1}^m \sum\limits_{\pi \in S_k} \frac{\omega_i}{1-x_i} \cdot M_\pi\left(\frac{1-x_i}{1+x_i}; A, E_1,\dots,E_k\right) \d t, \label{eq:invsqrt_stieltjes_quadrature}
\end{eqnarray}
where $x_i$ and $\omega_i$ denote the nodes and weights of an $m$-point Gauss--Jacobi quadrature rule with $\alpha=\beta=-\frac12$.

To illustrate the performance of this quadrature rule, we again consider the exact same experimental setup as in the previous example, just changing the function $f$. The results of this experiment are presented in Figure~\ref{fig:experiment_accuracy_invsqrt}. The Gauss--Jacobi quadrature rule~\eqref{eq:invsqrt_stieltjes_quadrature} rule produces an approximation with error near the order of machine precision for $m = 32$ quadrature nodes. Interestingly, in this experiment, the approximation for the higher-order \Fd{} actually converges faster than that for $A^{-1/2}\vb$.

\section{Test matrices from \texttt{anymatrix} collection}\label{sec:appendix_testset}
In this section, we give details on the set of test matrices used in Experiment~\ref{ex:anymatrix}. The data set was assembled as follows: 
\begin{enumerate}
    \item We selected all matrices with properties \texttt{square} and \replaced{\texttt{not rank-deficient}}{\texttt{non-singular}} from the \texttt{anymatrix} collection~\cite{Anymatrix}.
    \item We checked which of these matrices have eigenvalues on the negative real axis. For each such matrix $A$ we either \dots
    \begin{itemize}
        \item \dots replaced it with $-A$ in our test set if it had no eigenvalues on the positive real axis,
        \item \dots or removed it from the test set otherwise.
    \end{itemize}
    \item For all remaining matrices, we computed and compared approximations with algorithm \texttt{HR} in double and quadruple precision, using the small value $n = 25$ (or a value as close as possible to $25$ if the matrix was not freely scalable) due to the very high run time of the algorithm in quadruple precision. We then removed all matrices for which \texttt{HR} in double precision failed to produce an accurate enough result. This typically happened for matrices which were ``numerically singular'', i.e., whose smallest eigenvalue was near the order of unit roundoff.
\end{enumerate} 

Table~\ref{tab:anymatrix} contains an overview of the resulting test set, together with the relative accuracy that the different algorithms reached in Experiment~\ref{ex:anymatrix}. For matrices which required additional parameters in their construction, we chose the default parameters unless explicitly noted otherwise in the table. For the \texttt{contest} matrices which are based upon an input adjacency matrix, we used the tridiagonal adjacency matrix of a one-dimensional chain as input. \added{In the second part of the experiment, only the nonsymmetric matrices from the test set were used. Those are marked by an asterisk in Table~\ref{tab:anymatrix}.}

\begin{table}
\small
    \centering
    \begin{tabular}{l|r|r|r|r|p{5.75cm}} 
		& & \multicolumn{3}{c|}{\textbf{Relative accuracy}} &  \\ 
		\textbf{Matrix name} & \textbf{size $n$} & \textbf{\texttt{AA}} & \textbf{\texttt{Q}, $m=32$} & \textbf{\texttt{Q}, $m=96$} & \textbf{Properties \& Comments}\\ 
		\hline\hline
\texttt{gallery/wathen}     & $8$   & 6.5124e-08 &	0.0044168    &	2.2777e-15 & symmetric, constructed with parameters \verb|anymatrix('gallery/wathen', 1, 1);|\\ \hline
\multirow{2}{*}{\texttt{gallery/jordbloc}\added{${}^{\ast}$}}   & \multirow{2}{*}{$300$} & 7.0035e-10 &	3.3604e-15   &	3.4467e-15 & \multirow{2}{*}{non-normal}\\
& & \added{4.0570e-16${}^\ast$} & \added{1.8227e-07${}^\ast$} & \added{6.1653e-16${}^\ast$} &\\ \hline
\texttt{gallery/kms}        & $300$ & 3.7323e-10 &	2.323e-10    &	8.8975e-15  & symmetric \\ \hline
\multirow{2}{*}{\texttt{gallery/parter}\added{${}^{\ast}$}} & \multirow{2}{*}{$300$} & 2.3496e-09 &	1.0747e-14   &	1.0773e-14 & \multirow{2}{*}{non-normal}\\
& & \added{3.0772e-08${}^\ast$} & \added{1.3291e-14${}^\ast$} & \added{1.3223e-14${}^\ast$} &\\ \hline
\multirow{2}{*}{\texttt{contest/pagerank}\added{${}^{\ast}$}}  & \multirow{2}{*}{$300$} & 6.1561e-10 &	1.506e-14    &	1.5048e-14 & \multirow{2}{*}{\shortstack[l]{non-normal, constructed based on \\ 1D chain}}\\
& & \added{2.1954e-10${}^\ast$} & \added{1.2827e-14${}^\ast$} & \added{1.2712e-14${}^\ast$} &\\ \hline
\multirow{2}{*}{\texttt{core/rschur}\added{${}^{\ast}$}} & \multirow{2}{*}{$300$} & 4.2621e-09 &	1.5723e-14   &	1.5514e-14 & \multirow{2}{*}{non-normal} \\ 
& & \added{1.3995e-10${}^\ast$} & \added{5.7695e-16${}^\ast$} & \added{5.1198e-16${}^\ast$} &\\ \hline
\multirow{2}{*}{\texttt{core/kms\_nonsymm}\added{${}^{\ast}$}}  & \multirow{2}{*}{$300$} & 7.3768e-08 &	7.8759e-11   &	2.2121e-14 & \multirow{2}{*}{non-normal} \\
& & \added{4.5802e-09${}^\ast$} & \added{1.2669e-14${}^\ast$} & \added{1.2427e-14${}^\ast$} &\\ \hline
\texttt{gallery/condex}     & $300$ & 1.3934e-09 &	2.2884e-14   &	2.2759e-14 & symmetric \\ \hline
\multirow{2}{*}{\texttt{gallery/grcar}\added{${}^{\ast}$}}      & \multirow{2}{*}{$300$} & 5.3571e-10 &	3.4309e-14	 &  3.4346e-14 & \multirow{2}{*}{non-normal} \\
& & \added{1.5121e-08${}^\ast$} & \added{1.2546e-14${}^\ast$} & \added{1.2503e-14${}^\ast$} &\\ \hline
\texttt{core/hessmaxdet}\added{${}^{\ast}$}    & $300$ & 1.1233e-09 &	1.0742e-12	 &  4.4855e-14 & non-normal \\ 
& & \added{3.8555e-08${}^\ast$} & \added{1.0728e-10${}^\ast$} & \added{3.2312e-14${}^\ast$} &\\ \hline
\texttt{gallery/pei}        & $300$ & 7.0728e-10 &	4.8682e-14	 &  4.8769e-14 & symmetric \\ \hline
\texttt{gallery/poisson}    & $289$ & 9.5543e-10 &	6.2609e-11	 &  9.2545e-14 & symmetric \\ \hline
\texttt{core/wilson}        & $4$   & 4.7603e-09 &	0.0027599	 &  1.6243e-13 & symmetric \\ \hline
\multirow{2}{*}{\texttt{gallery/hanowa}\added{${}^{\ast}$}}     & \multirow{2}{*}{$300$} & 1.3245e-08 &	2.0481e-07   &  2.9535e-13  & \multirow{2}{*}{non-symmetric, normal}\\
& & \added{9.5941e-10${}^\ast$} & \added{4.9204e-16${}^\ast$} & \added{5.8360e-16${}^\ast$} &\\ \hline
\multirow{2}{*}{\texttt{gallery/invhess}\added{${}^{\ast}$}}    & \multirow{2}{*}{$300$} & 7.4889e-09 &	2.3603e-09	 &  8.0906e-13 & \multirow{2}{*}{non-normal}\\
& & \added{6.3634e-08${}^\ast$} & \added{6.0317e-10${}^\ast$} & \added{8.8928e-13${}^\ast$} &\\ \hline
\texttt{gallery/gcdmat}     & $300$ & 1.0726e-09 &	1.4861e-12	 &  1.6136e-12  & symmetric \\ \hline
\texttt{gallery/minij}      & $300$ & 3.8326e-09 &	7.0734e-11	 &  7.0734e-11 & symmetric \\ \hline
\multirow{2}{*}{\texttt{gallery/lesp}\added{${}^{\ast}$}}  & \multirow{2}{*}{$300$} & 2.1202e-07 &	7.2035e-07	 &  9.57e-11 & \multirow{2}{*}{non-normal} \\
& & \added{3.9038e-15${}^\ast$} & \added{5.2881e-15${}^\ast$} & \added{4.5620e-15${}^\ast$} &\\ \hline
\texttt{core/symmstoch}     & $300$ & 2.0727e-07 &	0.049153	 &  6.0204e-07  & symmetric \\ \hline
\texttt{gallery/lehmer}     & $300$ & 2.9657e-05 &	0.21587	     &  4.2459e-05 & symmetric \\ \hline
\texttt{gallery/randcorr}   & $300$ & 2.3464e-06 &	0.44885	     &  0.0014116  & symmetric \\ \hline
\texttt{gallery/tridiag}    & $300$ & 0.00024749 &	0.98876	     &  0.3893	 & symmetric \\ \hline
\multirow{2}{*}{\texttt{gallery/dorr}\added{${}^{\ast}$}}       & $300$ & 0.017097   &    1   	     &  0.98861 & \multirow{2}{*}{non-normal} \\
& & \added{1.1712e-01${}^\ast$} & \added{9.3916e-01${}^\ast$} & \added{6.4148e-01${}^\ast$} &\\ \hline
\multirow{2}{*}{\texttt{contest/mht}\added{${}^{\ast}$}}        & \multirow{2}{*}{$300$} & 0.01854	   &    1	         &  0.99858		 & \multirow{2}{*}{\shortstack[l]{non-normal, constructed based on \\ 1D chain}} \\
& & \added{5.1144e-08${}^\ast$} & \added{6.9387e-01${}^\ast$} & \added{4.6065e-01${}^\ast$} &\\ \hline
\texttt{regtools/deriv2}    & $300$ & 3.6029	   &    1	         &  1	  & symmetric \\ \hline
\multirow{2}{*}{\texttt{gallery/kahan}\added{${}^{\ast}$}}      & \multirow{2}{*}{$300$} & 1	   &    1	         &  1	 & \multirow{2}{*}{non-normal} \\
& & \added{9.6115e-16${}^\ast$} & \added{1${}^\ast$} & \added{1${}^\ast$} &\\ \hline
\texttt{gallery/wilk}       & $5$   & 0.1559	   &    1	         &  1	 & symmetric
\end{tabular}
\caption{Overview of the test set used in Experiment~\ref{ex:anymatrix}. \added{Matrices whose name is marked with an asterisk are also used in the second part of the experiment. For these matrices, the accuracies reached in the second part of the experiment are also depicted, marked with an asterisk as well.}}
\label{tab:anymatrix}
\end{table}

\end{document}